\documentclass{amsart}
\usepackage{amssymb,graphicx,rotating,multirow,multicol}
\usepackage{amsmath,amsthm,amsfonts,amscd}
\usepackage[toc,page]{appendix}
\usepackage{caption,comment,marginnote}
\usepackage[mathscr]{euscript}
\usepackage{epstopdf}
\usepackage{float}
\usepackage{graphicx}
\usepackage{hyperref}

\usepackage{epigraph}
\swapnumbers
\newtheorem{theorem}{Theorem}[subsection]

\newtheorem{lemma}[theorem]{Lemma}

\newtheorem{cor}[theorem]{Corollary}
\newtheorem{proposition}[theorem]{Proposition}
\theoremstyle{remark}

\numberwithin{equation}{subsection}
\newcommand{\K}{\ensuremath{\mathbb{K}}}

{\catcode`\@=11 \gdef\mnote#1{\marginpar{\tiny
 \tolerance\@M\spaceskip2.6\p@ plus10\p@ minus.9\p@\rm#1}}}
\def\Dg:{\endgraf{\bf Dg:\enspace}\ignorespaces}

\let\Bbb\mathbb
\let\Cal\mathcal
\def\S{\mathbb S}
\def\sm{\smallsetminus}
\def\Card{\operatorname{Card}}

\newcommand{\be}{\begin{equation}}
\newcommand{\ee}{\end{equation}}

\def\Z{\Bbb Z}
\def\R{\Bbb R}
\def\C{\Bbb C}

\def\P{\Bbb P}
\def\T{\Bbb T}
\def\B{\mathcal B}
\def\CC{\mathcal C}
\def\V{\mathcal V}
\def\E{\mathcal E}

\def\Rp#1{\Bbb{RP}^{#1}}

\def\Pin{\operatorname{Pin}}
\def\rk{\operatorname{rk}}

\def\W{\operatorname{\Lambda}}

\def\conj{\operatorname{conj}}

\def\q{\operatorname{\widehat q}}
\def\Pic{\operatorname{Pic}}

\def\cs{\til{\frak s}}
\def\css{\frak s}

\def\e{\varepsilon}
\def\s{\mathbf s}
\def\ee{\boldsymbol{\ell}}

\def\dsum{\bot\!\!\!\bot}
\let\+\dsum
\let\a=\alpha
\let\ge\geqslant 
\let\le\leqslant 
\let\la\langle
\let\ra\rangle
\let\til\widetilde

\makeatletter
\newcommand{\addresseshere}{%
  \enddoc@text\let\enddoc@text\relax
}
\makeatother

\title[Combined count of real rational curves of canonical degree 2]{
Combined count of real rational curves of canonical degree 2 on real del Pezzo surfaces with $K^2=1$}
\author[]
{S.~Finashin, V.~Kharlamov}
\makeatletter
\@namedef{subjclassname@2020}{%
  \textup{2020} Mathematics Subject Classification}
\makeatother


\begin{document}
\begin{abstract}
We propose two systems of "intrinsic" signs for counting such curves. In both cases the result acquires an exceptionally strong invariance property: it does not depend on
the choice of a surface. One of our counts includes all divisor classes of canonical degree 2 and gives in total 30. The
other one excludes the class $-2K$, but adds up the results of counting for a pair of real structures that differ by Bertini involution. This count gives 96.
\end{abstract}

\maketitle

\setlength\epigraphwidth{.63\textwidth}
\epigraph{
Vor dem Gesetz steht ein Türhüter. \\
Zu diesem Türhüter kommt ein Mann vom Lande und bittet um Eintritt in das Gesetz.
Aber der Türhüter sagt, daß er ihm jetzt den Eintritt nicht gewähren könne.
Der Mann überlegt und fragt dann,
ob er also später werde eintreten dürfen.\\
"Es ist möglich", sagt der Türhüter, "jetzt aber nicht."}{F.~Kafka, Die Parabel "Vor dem Gesetz" \\}


\section{Introduction}
This work is based on our previous paper \cite{TwoKinds}.
So we start with recalling
its setup and principal ingredients.

\subsection{Short review of \cite{TwoKinds}}
By definition, a compact complex surface $X$ is a {\it del Pezzo surface of degree $1$},
if $X$ is non-singular and irreducible,
its anticanonical class $-K_X$ is ample,
and $K^2_X=1$. The image of $X$ by a bi-anticanonical
map $X\to \P^3$ is then
a non-degenerate quadratic cone $Q\subset \P^3$, with $X\to Q$ being
a double covering branched at the vertex of the cone
and along a non-singular sextic curve $C\subset Q$
(a transversal intersection of $Q$ with a cubic surface).
In particular,
each del Pezzo surface of degree 1 carries a non-trivial automorphism, known as  the {\it Bertini involution}, that is the deck transformation $\tau_X$ of the covering.

Any real structure, $\conj:X\to X$,
has to commute with $\tau_X$, and this gives another real structure $\tau_X\circ\conj=\conj\circ\tau_X$ which we call {\it Bertini dual} to $\conj$.
It is such a pair of real structures, $\{\conj,\ \conj\circ\tau_X\}$,
that we call a {\it Bertini pair}. We generally use notation
$\conj^\pm$ for Bertini pairs
of real structures
and write $X^\pm$ for the corresponding pairs of real del Pezzo surfaces to simplify a more formal notation $(X,\conj^\pm)$.

The bi-anticanonical map projects the real loci $X^\pm_\R$
to two complementary domains $Q^\pm_\R\subset Q_\R$
on $Q_\R$, where the latter
is a cone over a real non-singular conic \emph{with non-empty real locus.}
The branching curve $C$ is real too,
and its real locus $C_\R$, together with the vertex of the cone, forms the common boundary of $Q^\pm_\R$.
Conversely, for any real non-singular curve $C\subset Q$ which is a transversal intersection of $Q$ with a real cubic surface, the surface $X$ which is the double covering of
$Q$ branched at the vertex of $Q$ and along $C$ is a del Pezzo surface of degree $1$ inheriting from $Q$ a pair of
Bertini dual real structures $\conj^\pm$.

Recall also
an intrinsic description of the basic $Pin^-$-structure introduced in \cite{TwoKinds}.
 \begin{theorem}\label{main-Pin}
There is a unique way to supply each real del Pezzo surface $X$ of degree 1
with a $\Pin^-$-structure $\theta_{X}$ on $X_\R$, so that the following properties hold:
\begin{enumerate}
\item\label{invar}
$\theta_X$
is invariant under real automorphisms and real deformations of $X$. In particular, the associated quadratic function
$q_{X} : H_1(X_\R;\Z/2)\to\Z/4$ is preserved by the Bertini involution.
\item\label{vanish}
$q_{X}$ vanishes on each real vanishing cycle in $H_1(X_\R;\Z/2)$ and
takes value $1$ on the class dual to $w_1(X_\R)$.
\item\label{sym}
If $X^\pm$ is
a Bertini pair of real del Pezzo surfaces of degree 1, then
the corresponding quadratic functions $q_{X^\pm}$
take equal values on the elements represented
in $H_1(X^\pm_\R;\Z/2)$  by the connected components of $C_\R$.
\end{enumerate}
\end{theorem}
(Here, by real vanishing cycles in $H_1(X_\R;\Z/2)$ we understand cycles realized by loops pinched under nodal degenerations.)

The Picard group of a del Pezzo surface (as well as that of any rational surface) is naturally isomorphic to the second homology group with integer coefficients, $\Pic X=H_2(X)$.
It has a natural grading by
{\it canonical degree}, $\alpha\mapsto -\alpha\cdot K_X$.
In the case of del Pezzo surfaces of degree 1 the lattice $K_X^\perp\subset H_2(X)$ of elements of
canonical degree $0$
is isomorphic to $E_8$.
If $X$ is real,
then the Picard group of real divisor classes is naturally isomorphic to
$$ H_2(X)\cap \ker(1+\conj_*)=\Z K_X\oplus \W(X), \quad \W(X):= K_X^\perp \cap \ker(1+\conj_*).$$
For the list of
isomorphism classes of the
lattices $\W(X)$,
see Table 1 in Section \ref{max-submax}.

None of the divisor classes of degree 0 is effective, and
the only effective divisor classes
of degree 1 are $-K_X$ and $-K_X-e$ where $e$ is any {\it root} (an element of square $-2$) of $K_X^\perp=E_8$.
The linear system $\vert -K_X\vert$ is a pencil of curves
formed by pull-backs of
the line generators of $Q$.
Each of the classes $-K_X-e$ has a unique representation by an effective divisor. These effective divisors are $(-1)$-curves and, in the context of del Pezzo surfaces,
are called {\it lines}.
Over $\C$
the lines are in one-to-one correspondence with the roots
$e$ in $K_X^\perp= E_8$,
and over $\R$ with the roots in $\W(X)$.

It is the $\Pin^-$-structure of Theorem \ref{main-Pin} that opened a way to a {\it signed count} of real lines in \cite{TwoKinds}
where we introduced
two species of real lines $L\subset X$
distinguished by the values of $q_{X}([L_\R])\in\Z/4$. Namely, we called a real line $L\subset X$ {\it hyperbolic} (resp. {\it elliptic}) if
$q_{X}([L_\R])=1\in\Z/4$ (resp. $q_X([L_\R])=-1\in\Z/4$), and attributed to
it an
{\sl integer weight} $\css(L)=\pm1\in\Z$, $\css(L)\mod 4=q_{X}([L_\R])$.
As was shown in \cite{TwoKinds}, the count of real lines with these weights gives
the following fundamental relations:

\begin{equation}\label{rank-formula}
\sum_{L\in\mathcal L_\R(X)} \css(L) = 2 \rk \W(X),
\end{equation}
where $\mathcal L_\R(X)$ denotes the set of real lines in $X$; and
\begin{equation}\label{16}
\sum_{L\in \mathcal L_\R(X^+)} \css(L) + \sum_{L\in\mathcal L_\R(X^-)} \css(L) =16,
\end{equation}
where $X^\pm$ is any Bertini pair.

The divisor classes
of lines together with $-K_X$ constitute the {\it first layer}
in the semigroup of effective divisor classes, that is the set of effective divisor classes
$\alpha$ of canonical degree $-\alpha\cdot K_X=1$.
Therefore, it is natural
to combine the count of real lines with a count of real rational curves $A$ belonging to the divisor class  $-K_X$.
For a generic $X$, all these real rational curves $A$ are nodal and we (similar to the above) attribute  to each of them the weight
$\css(A)=i^{q_X([A_\R])-1}w(A)$,
where $w(A)=(-1)^{c_A}$ is the {\it modified Welschinger weight} of $A$, with
$c_A$ standing for the number of cross-point real nodes of $A$.

The linear system $\vert -K_X\vert $ is a pencil of curves of arithmetic genus 1 with one base point. Thus, we can count the Euler characteristic of $X_\R$ blown up at this base point by means of this singular elliptic pencil, and obtain
$$
\sum_{\text{\rm real rational\, } A\in \vert -K_X\vert } \css(A)=
\sum_{\text{\rm real rational\, } A\in \vert -K_X\vert } (-1)^{c(A)}=\chi(X_\R)-1,
$$
since, for each $A\in \vert -K_X\vert$, the class $[A_\R]$ is dual to $w_1(X_\R)$, and thus $q_X[A_\R])=1$ (see Theorem \ref{main-Pin}).
On the other hand, from the Lefschetz fixed point formula, $\chi(X_\R)-1=(8-\rk \W(X))-\rk \W(X)$. Combining the above
with the relation
(\ref{rank-formula}), we achieve the following result.

\begin{theorem}\label{8}
For any generic
real del Pezzo surface $X$ of degree $1$,
\begin{equation}\label{eq8}\pushQED{\qed}
\sum_{L\in\mathcal L_\R(X)} \css(L) + \sum_{\text{\rm real rational\, } A\in \vert -K_X\vert } \css(A) =8.
\qedhere
\popQED
\end{equation}
\end{theorem}

An assumption of genericness stands here only to ensure that all
rational curves in the pencil $\vert -K_X\vert $ are nodal. This assumption can be excluded by interpreting $c_A$ as a "virtual" number of cross-point real nodes.

\subsection{Next step}
In this paper, we study in a similar way the {\it second layer},
that is the set
of effective divisor classes
$\alpha$ with $-\alpha K_X=2$.
Each of such
$\a$ is a sum of two elements from the first layer, since
the elements of the first layer (that is $-K$ and classes of lines)
generate the
whole semigroup of effective divisor classes.
To extend the
relations (\ref{16}) and (\ref{eq8}) to the second layer, we exclude from consideration the {\it double line} divisor classes
$-2K_X-2e$ where $e$ is a root in $K_X^\perp$, and the classes of type $-2K_X-e_1-e_2$ where $e_1,e_2$ are roots in $K_X^\perp$ with $e_1\cdot e_2 =-1$.
This is motivated by the fact that their Gromov-Witten invariants are zero, since none of these classes $\alpha$ contains a rational curve passing
through $1=-\alpha K_X-1$ fixed generic point.

So, the remaining piece
$\B(X)$ of the second layer and its {\it real part}
$\B_\R(X)=\B(X)\cap \ker(1+\conj_*)$
splits as

\begin{equation*}
\begin{aligned}
\B(X)=&\B^0(X)\cup\B^2(X)\cup\B^4(X),\quad \B^{2k}(X)=\{-2K_X-v \,|\, v\in K_X^\perp, v^2=-2k\}, \\
\B_\R(X)=&\B^0_\R(X)\cup\B^2_\R(X)\cup\B^4_\R(X),\quad \B_\R^{2k}(X)=\B^{2k}(X)\cap \ker(1+\conj_*).
\end{aligned}
\end{equation*}

The curves $A\subset X$ in each of the divisor classes $\alpha\in \B^{2k}$ have arithmetic genus $2-k$ and form a linear system of projective dimension $3-k$.
Thus, we pick a point $x\in X_\R$ and for each $\a\in\B^{2k}$
consider the following sets of curves

\begin{equation*}
\begin{aligned}
\CC^{2k}(X,\a,x)=&\{A\subset X\,|\, A \text{ is rational}, [A]=\a,x\in A\}, \\
\CC^{2k}_\R(X,\a,x)=&\{A\in \CC^{2k}(X,\a,x)\,|\, A \text{ is real}\}
\end{aligned}
\end{equation*}
and  put
\begin{equation*}
\begin{aligned}
 &\CC^{2k}(X,x)=\bigcup_{\a\in  \B^{2k}(X)}\CC^{2k}(X,\a,x),
 \quad
 \CC_\R^{2k}(X,x)=\bigcup_{\a\in  \B_\R^{2k}(X)}\CC_\R^{2k}(X,\a,x)\\
& \CC(X,x)= \CC^{0}(X,x)\cup\CC^{2}(X,x) \cup \CC^{4}(X,x),
 \quad
 \CC_\R(X,x)= \CC^{0}_\R(X,x)\cup\CC_\R^{2}(X,x) \cup \CC_\R^{4}(X,x).
\end{aligned}
\end{equation*}
For a generic point $x\in X_\R$, each of these sets is finite,
and each of the curves in these sets is either non-singular or nodal.

The main results of this paper are the following two theorems, which provide an extension of the strong invariance properties (\ref{16}) and (\ref{eq8}) from the first
layer to the second.

\begin{theorem}\label{main-30} For any real del Pezzo surface $X$ of degree 1
and any generic point $x\in X_\R$, we have
\begin{equation}\label{second_weight}
\sum_{A\in  \CC_\R(X,x)}\css(A)=30
\end{equation}
with
\begin{equation}\label{s-weight-definition}
\css(A) = i^{\q_X([A])}w(A), \quad w(A)=(-1)^{c_A}
\end{equation}
where $\q_X([A])=q_X(A_\R)$ and $c_A$ stands for the number of cross-point  real nodes of $A$.
\end{theorem}

 \begin{theorem}\label{main}
For any Bertini pair of real del Pezzo surfaces $X^\pm$ of degree 1 and any pair of real generic points $x^\pm\in X^\pm_\R$, we have
\begin{equation}
\sum_{A\in  \CC_\R^2(X^+,x^+)\cup\CC_\R^4(X^+,x^+)}\cs(A) +\sum_{A\in  \CC_\R^2(X^-,x^-)\cup\CC_\R^4(X^-,x^-)}\cs(A) =96
\end{equation}
where
\begin{equation}\label{ss-weight-definition}
\cs(A) = \begin{cases}
\css(A), \text{ if } A\in\CC_\R^2(X^\pm,x^\pm),\\
2\css(A), \text{ if } A\in\CC_\R^4(X^\pm,x^\pm).\\
\end{cases}
\end{equation}

 \end{theorem}

\subsection{Motivations and order of presentation in the paper}
It may be worth mentioning that our initial motivation was
to study quadric sections of a real quadric cone $Q\subset \P^3$ that are 6-tangent to a fixed real sextic curve
$C\subset Q$
and to elaborate for them a count
which would have as strong invariance properties as the count of real 3-tangent
hyperplane sections
established  in \cite{TwoKinds}
as one of the main applications of the relation (\ref{16}) (see below Subsection \ref{6-tangents}). It is
from analysis of the underlying wall-crossing phenomena that we came to an idea to combine together
$\B^{2}(X)$ and  $\B^{4}(X)$ and developed the corresponding system of weights.
In this way
we arrived to Theorem \ref{main} and only
later on elaborated another, more arithmetic
proof.
We present here the both proofs believing that they both should help to reveal a general law.
At least,
it is this arithmetic proof that led us to an observation that by changing the system of weights and extending the count to $\B^0(X)$
we come to Theorem \ref{main-30}, where contrary to
Theorem \ref{main},
invariance of
count is achieved for each real structure separately (without coupling them into Bertini pairs).

The paper is organized as follows. The arithmetic proof is presented in Sections \ref{max-submax} and \ref{Proofs}.
Namely, in Section \ref{max-submax} we treat separately a bit more tricky case of maximal and submaximal surfaces,
while the other cases are considered
in Section \ref{Proofs}. An alternative
proof,
via analysis of wall-crossing, is presented
in Section \ref{wall-crossing}. In Section \ref{Concluding},
as an application of Theorem \ref{main},
we perform a signed count of sections $B=Q\cap Z$ by quadrics $Z\subset \P^3$
which are 6-tangent to a given real sextic $C\subset Q$,
and also make a few remarks on a symplectic perspective.

\subsection{Acknowledgements} This work was finalized
during staying of the second author at the Max-Planck Institute for Mathematics in Bonn.
He thanks the MPIM for hospitality and excellent  (despite a complicated pandemic situation) working conditions.
The second author was partially supported by the grant ANR-18-CE40-0009 of Agence Nationale de Recherche
and the grant by Ministry of Science and Higher Education of Russia under the
contract 075-15-2019-1620 with St. Petersburg Department of Steklov Mathematical Institute.

We thank the referee for valuable remarks and suggestions.


\section{Preliminary count for surfaces with a connected maximal or submaximal real locus}\label{max-submax}

\subsection{Real forms of del Pezzo surfaces of degree 1 {\sl {\rm (}see \cite{TwoKinds} and references therein{\rm )}}}
Recall that the real deformation class of any real del Pezzo surface $X$ of degree 1 is determined by the topology of $X_\R$.
There are 11 deformation classes.
The corresponding topological types are shown in the first line of Table \ref{eigenlattices}, where $\T^2$ stands for a 2-torus  and $\K$ for a Klein bottle.

The lattice $\W(X)=K_X^\perp \cap \ker(1+\conj_*)$ is one of the main deformation invariants which plays a crucial role in the further proofs.
These lattices are enumerated in the
bottom
line
of Table \ref{eigenlattices}.

This table is organized according to the so-called {\it Smith type} of surfaces, which is denoted by $(M-r)$ and indicates
that
in the Smith inequality, $\dim H_*(X_\R;\Z/2)\le \dim H_*(X;\Z/2)$, the right-hand side is greater by
$2r$ than the left-hand side.
The $(M-2)$-case includes four deformation classes and
two of them, encoded with $(M-2)_I$,
are {\it of type $I$}, which means that the fundamental class of $X_\R$ is realizing
$w_2(X)=K_X({\rm mod}\, 2)$ in $H_2(X;\Z/2)$.

Surfaces belonging to the same real Bertini pair have the same Smith type. The real Bertini pairs
form 7 pairs of real deformation types.
In 3 pairs the deformation types (indicated in the last 3 columns of Table \ref{eigenlattices})
are dual to itself: $X^+$ is deformation equivalent to $X^-$.
The other 4 pairs are shown in the
4 columns marked M, (M-1), (M-2), and (M-3).

Since $\conj_-=\tau_X\circ \conj_+$ and $\tau_X$ acts in $H_2(X)$ as multiplication by $-1$ in $K^\perp_X$,  the lattice $\W(X^-)= K_X^\perp\cap \ker (1+\conj_*^-)$ coincides with
$\ker (1+\conj_*^+)= K_X^\perp\cap \ker (1+\conj_*^+)$.
In particular, the lattices $\W(X^+)$ and $\W(X^-)$ are orthogonal complements to each other in $E_8=K_X^\perp$.

\begin{table}[h]
\caption{The root lattices
$\W(X)=K_X^\perp \cap \ker(1+\conj_*)$}\label{eigenlattices}
\resizebox{\textwidth}{!}{
\hbox{\boxed{\begin{tabular}{c||ccccccc}
Smith type of $X_\R$&$M$&$(M-1)$&$(M-2)$&$(M-3)$&$(M-4)$&$(M-2)_I$&$(M-2)_I$\\
\hline\hline
Topology of $X_\R$&$\Rp2\#4\T^2$&$\Rp2\#3\T^2$&$\Rp2\#2\T^2$&$\Rp2\#\T^2$&$\Rp2$&$\Rp2\dsum \K^2$& $(\Rp2\#\T^2)\dsum \S^2$\\
\hline
$\W(X)$&$E_8$&$E_7$&$D_6$&$D_4+A_1$&$4A_1$&$D_4$&$D_4$\\
\end{tabular}}}}
\resizebox{0.315\textwidth}{!}{
\hskip-57.3mm\hbox{\boxed{\begin{tabular}{c||cccc}
Smith type of $X_\R$&$M$&$(M-1)$&$(M-2)$&$(M-3)$\\
\hline\hline
Topology of $X_\R$&$\Rp2\dsum 4\S^2$&$\Rp2\dsum 3\S^2$&$\Rp2\dsum 2\S^2$&$\Rp2\dsum \S^2$\\
\hline
$\W(X)$ &0&$A_1$&$2A_1$&$3A_1$\\
\end{tabular}}}}
\end{table}

\subsection{Cremona transformation of Pin-codes}\label{Pin-codes}
Consider a {\it real blowup model} $X\to \P^2$ of a real del Pezzo
$(M-r)$-surface $X$ of degree 1 with
 $r$ pairs of complex conjugate imaginary exceptional classes $\ee_{8-2k}=-\conj_* \ee_{8-2k-1}, 0\le k\le r-1$,
 and $8-2r$ real
exceptional classes $\ee_1,\dots,\ee_{8-2r}$.
By the {\it code} of such model
we mean a sequence $(a_0,\dots,a_{8-2r})$ of residues $\pm1\mod4$, where
$a_i=\q_X(\ee_i)$ for $i\ge 1$ and $a_0=\q_X(h)$ with $h$ staying for the class realized by the pull-back of straight lines.
The condition
$\q_X(h+\ell_1+\dots+\ell_{8-2r})=q_X(w_1(X_\R))=1$ (see Theorem \ref{main-Pin}(2))
imposes the relation
\begin{equation}\label{code-relation}
a_0+\dots+a_{8-2r}=1\mod4.
\end{equation}

\begin{lemma}\label{Cremona}
An elementary Cremona transformation based on a triple $\ee_i,\ee_j,\ee_k$
with $1\le i< j<k\le 8-2r$
changes each of the residues $a_0,a_i,a_j,a_k$
to the sum of three others, and does not change $a_l$ for $l\ne0,i,j,k$.
In particular:
 a sequence  $a_0,a_i,a_j,a_k$ formed by $1,1,1,1$ is replaced by $-1,-1,-1,-1$ and vice versa;
a sequence $1,1,-1,-1$ is
replaced by
$-1,-1,1,1$ and vice versa;
sequences $1,1,1,-1$ and $-1,1,1,1$ are not modified.

If $r>0$ and we
perform an elementary Cremona transformation based on
a triple $\ee_i,\ee_7,\ee_8$ (where $\ee_7,\ee_8$ are conjugate imaginary), then the pair $a_0,a_i$
is changed to $a_i,a_0$
while the other elements of the code are not modified.
\end{lemma}

\begin{proof} Such transformation changes  $\ee_i$ to $h-\ee_j-\ee_k$ and $h$ to $2h-\ee_i-\ee_j-\ee_k$, and the result follows from quadraticity of $\q_X$
and its additivity
on pairwise orthogonal elements.
\end{proof}

\subsection{Real rulings on del Pezzo surfaces of degree 1}\label{rulings}
\begin{proposition}\label{rulings-prop} 
Let $X$ be a del Pezzo surface of degree 1. If $\a\in  \B^{4}(X)$ $($respectively, $\a\in  \B^{2}(X)$\!$)$, then the linear system $\vert \a\vert$ is a pencil $($respectively, a net$)$ of rational curves $($respectively, curves of arithmetic genus 1$)$ without fixed components and base points.
\end{proposition}

\begin{proof} Let $\a=-2K-v\in  \B^{4}(X)$, $v^2=-4$. Then, by Riemann-Roch theorem, $\dim H^0(X,\Cal O_X(\a))\ge \frac12\a^2-\frac12\a K+1-\dim H^0(X,\Cal O_X(K-\a))=
\frac12{\a^2}-\frac12\a K+1= 2$ (here, the vanishing of  $H^0(X,\Cal O_X(K-\a))$ is due $-K(K-\a)=-1<0$). By adjunction, $g_a(\a)=\frac{\a^2+\a K}2+1=0$. Since $-K$ is ample and $-K\alpha= 2$, we conclude that
each curve $D\in\vert \a\vert $ is either an embedded nonsingular irreducible rational curve with $D^2=0$, or $D=E_1+E_2$ where $E_1,E_2$ are $(-1)$-curves with $E_1\cdot E_2=1$. This implies that a generic $D\in\vert \a\vert $ is of the first kind, and, hence, the linear system is
a pencil without fixed components and base points. 

Now, let $\a=-2K-v\in  \B^{2}(X)$, $v^2=-2$. Consider the line $L$ of divisor class $-K-v$ and the pull backs $D_1,D_2$ of two distinct generators of $Q$ with respect to the double covering $X\to Q$. 
Note that 
$L+D_1+D_2\in \vert -K+\a\vert$ is a so-called 1-connected divisor
(pairwise intersections of its 3 components are $\ge 1$). Thus, due to Ramanujam's vanishing theorem and Riemann-Roch, we have
$\dim H^0(X,\Cal O_X(\a))= \frac12\a^2-\frac12\a K+1-\dim H^0(X,\Cal O_X(K-\a))+\dim H^1(X,\Cal O_X(K-\a)) =
\frac12{\a^2}-\frac12\a K+1= 3$. By adjunction, $g_a(\a)=\frac{\a^2+\a K}2+1=1$. 
To check that the net $\vert \a\vert$ has no fixed components and base points, 
it is sufficient to notice that 
it contains a sub-pencil of reducible curves $L+D$ 
where $D\in \vert -K\vert$
are pull-backs of generators of $Q$, and 
to restrict  
the net to 
any of these reducible curves. 
\end{proof}

The following corollary is straightforward.
\begin{cor}\label{pencils} 
Let $X$ be a real del Pezzo surface of degree 1.
\begin{enumerate}
\item If $\a\in  \B_\R^{4}(X)$, then for any $x\in X_\R$ there exists one and only one curve $A\in \CC^{4}_\R(X,\a,x)$.
\item If $\a\in  \B_\R^{2}(X)$, then for any generic $x_1\in X_\R$, the curves $A\in\vert \a\vert $ passing through $x_1$ form a real pencil having a second fixed point $x_2\ne x_1$, and each singular curve in this pencil is nodal.
\qed
\end{enumerate}
\end{cor}

\subsection{Signed count in the case of connected M-surfaces}
\begin{proposition}\label{M-model}
If $X_\R= \Rp2\#4\T^2$, then $X$ admits
a real blowup model with 8 real blown up points and code $(1,1,\dots, 1)$.
\end{proposition}

\begin{proof}
Let us blow up $\P^2$
first at 4 generic points $p_1,\dots,p_4$ and next make a generic infinitely near blowup over
each of the points $p_i$.
The result is a singular del Pezzo surface of degree 1 with 4 nodes.
A non-singular del Pezzo surface $X$
obtained by its perturbation
can be interpreted as replacing of the 4 infinitely near blow ups
by blowing up at points $p_{i+4}\in \P^2$ located somewhere in close proximity to $p_i$, $i=1,\dots,4$.
Let $\ee_i\in H_2(X)$ denote the exceptional
class
 of blowing up at $p_i$.
In the real setting, our assumption $X_\R= \Rp2\#4\T^2$ means that
all points $p_i, 1\le i\le 8$, are real.

Note moreover, that $\q_X(\ee_{i})$ and $\q_X(\ee_{i+4})$ are of opposite signs,
since $\ee_{i}-\ee_{i+4}$ is a vanishing class.
According to (\ref{code-relation}),
this implies $\q_X(h)=1$.
Then, as it follows from Lemma \ref{Cremona},
an elementary Cremona transformation based at two negative and one positive classes
$\ee_i$ leads to a
real
blowup model with totally 3 negative classes $\ee_i$.
After another transformation based at these three, we obtain
a real blowup model with code $(1,1,\dots, 1)$, as required.

To extend the result from a particular surface $X$ (constructed above) to any other real del Pezzo surface $X'$ of degree 1 with the real locus $X'_\R=\Rp2\#4\T^2$, it is sufficient to use their real deformation equivalence, the invariance
of the quadratic function $\q$ under real deformation, and the natural bijection between the set of real $(-1)$-curves and the set of divisor classes $\alpha$ with $\alpha K=\alpha^2=-1$.
\end{proof}

In what follows, for a given
surface $X$, we use notation
$$
R^{2k}(X)=\{v\in K_X^\perp\,|\, v^2=-2k \}, k\in \Z.
$$
If $X$ is real then we consider the real counterpart of the above
sets of divisor classes and put
$$
R_\R^{2k}(X)=\{x\in R^{2k}(X)\,|\,\conj_*(x)=- x\}.
$$

As is well known
(and indicated in Table \ref{eigenlattices}),
 if $(X,\conj)$ is a maximal real del Pezzo surface of degree 1 and $X_\R$ is connected, then
$X_\R= \Rp2\#4\T^2$ and
$\W(X)=K_X^\perp$,
so that in this case $R_\R^{2}(X)=R^{2}(X)$ is nothing but the set of roots in
$K_X^\perp
=E_8$.
To enumerate the elements of this set and to determine their
$\q_X$-values, we use the special blowup model given by Proposition \ref{M-model},
which we call a {\it positive blowup model}.

A straightforward calculation shows that with respect to a positive blowup model the 240 roots that constitute $R_\R^{2}(X)$ in the case
$X_\R= \Rp2\#4\T^2$ split into 4 types with corresponding values of $\q_X$ as shown in Table \ref{M-roots}\footnote{The presence of real roots with $\hat q\ne 0$ demonstrates existence of real roots that can not be realized by a vanishing cycle of a real nodal degeneration. In fact, such real roots exist
on each real del Pezzo surface of degree 1 with $X_\R$ containing a component of non-positive Euler characteristic, see \cite{TwoKinds}.}.
Each type is characterized there by its {\it level}
equal
(up to sign) to the coefficient at $h$ in the basic coordinate expansion.

\begin{table}[h]
\caption{Real roots in the case $X_\R= \Rp2\#4\T^2$}\label{M-roots}
\centerline{
\scalebox{0.8}{\boxed{\begin{tabular}{c|c|c|c}
level&type of roots $e \in 
R_\R^{2}(X)\cong E_8$&number&$\q_X(e)$\\
\hline
0&$\ee_{i_1}-\ee_{i_2}$&56&2\\
1&$\pm(h-\ee_{i_1}-\ee_{i_2}-\ee_{i_3})$&$2\binom83=112$&0\\
2&$\pm(2h-\ee_{i_1}-\ee_{i_2}-\ee_{i_3}-\ee_{i_4}-\ee_{i_5}-\ee_{i_6})$&$2\binom86=56$&2\\
3&$\pm(3h-2\ee_{i_1}-\ee_{i_2}-\ee_{i_3}-\ee_{i_4}-\ee_{i_5}-\ee_{i_6}-\ee_{i_7}-\ee_{i_8})$&16&0\\
\end{tabular}}}}
\end{table}

As a consequence the 240 elements in $\B_\R^{2}(X)$
split into 7 types with corresponding values of $\q_X$ as shown in Table \ref{M-beta}.

\begin{table}[h]
\caption{$\B_\R^{2}$-classes  in the case $X_\R= \Rp2\#4\T^2$}\label{M-beta}
\centerline{
\scalebox{0.8}{\boxed{\begin{tabular}{c|c|c|c}
level&type of classes
$\alpha\in \B_\R^{2}(X)$&number&$
\q_X(\alpha)$\\
\hline
3&$3h-\ee_{i_1}-\ee_{i_2}-\ee_{i_3}-\ee_{i_4}-\ee_{i_5}-\ee_{i_6}-\ee_{i_7}$&8&0\\
4&$4h-2\ee_{i_1}-2\ee_{i_2}-\ee_{i_3}-\ee_{i_4}-\ee_{i_5}-\ee_{i_6}-\ee_{i_7}-\ee_{i_8}$&28&2\\
5&$5h-2\ee_{i_1}-2\ee_{i_2}-2\ee_{i_3}-2\ee_{i_4}-2\ee_{i_5}-\ee_{i_6}-\ee_{i_7}-\ee_{i_8}$&$56$&0\\
6&$6h-3\ee_{i_1}-2\ee_{i_2}-2\ee_{i_3}-2\ee_{i_4}-2\ee_{i_5}-2\ee_{i_6}-2\ee_{i_7}-\ee_{i_8}$&56&2\\
7&$7h-3\ee_{i_1}-3\ee_{i_2}-3\ee_{i_3}-2\ee_{i_4}-2\ee_{i_5}-2\ee_{i_6}-2\ee_{i_7}-2\ee_{i_8}$&$56$&0\\
8&$8h-3\ee_{i_1}-3\ee_{i_2}-3\ee_{i_3}-3\ee_{i_4}-3\ee_{i_5}-3\ee_{i_6}-\ee_{i_7}-\ee_{i_8}$&28&2\\
9&$9h-4\ee_{i_1}-3\ee_{i_2}-3\ee_{i_3}-3\ee_{i_4}-3\ee_{i_5}-3\ee_{i_6}-3\ee_{i_7}-3\ee_{i_8}$&8&0\\
\end{tabular}}}}
\end{table}

Next, we consider $\B^4_\R$. As is known (see \cite[Table 4.9]{Conway}), 
its number of elements is $2160$.
 All  2160 elements are listed in Table \ref{M-gamma}.


\begin{table}[h]
\caption{$\B_\R^{4}$-classes  in the case $X_\R= \Rp2\#4\T^2$}\label{M-gamma}
\scalebox{0.8}{\boxed{\begin{tabular}{c|c|c|c}
level&
type of classes $\beta\in \B_\R^{4}(X)$ &number&$\q_X(\beta)$\\
\hline
1&$h-\ee_{i}$&8&2\\
2&$2h-\ee_{i_1}-\ee_{i_2}-\ee_{i_3}-\ee_{i_4}$&70&0\\
3&$3h-2\ee_{i_1}-\ee_{i_2}-\ee_{i_3}-\ee_{i_4}-\ee_{i_5}-\ee_{i_6}$&168&2\\
4&$4h-2\ee_{i_1}-2\ee_{i_2}-2\ee_{i_3}-\ee_{i_4}-\ee_{i_5}-\ee_{i_6}-\ee_{i_7}$&280&0\\
4&$4h-3\ee_{i_1}-\ee_{i_2}-\ee_{i_3}-\ee_{i_4}-\ee_{i_5}-\ee_{i_6}-\ee_{i_7}-\ee_{i_8}$&8&0\\
5&$5h-2\ee_{i_1}-2\ee_{i_2}-2\ee_{i_3}-2\ee_{i_4}-2\ee_{i_5}-2\ee_{i_6}-\ee_{i_7}$&$56$&2\\
5&$5h-3\ee_{i_1}-2\ee_{i_2}-2\ee_{i_3}-2\ee_{i_4}-\ee_{i_5}-\ee_{i_6}-\ee_{i_7}-\ee_{i_8}$&$280$&2\\
6&$6h-3\ee_{i_1}-3\ee_{i_2}-2\ee_{i_3}-2\ee_{i_4}-2\ee_{i_5}-2\ee_{i_6}-\ee_{i_7}-\ee_{i_8}$&420&0\\
7&$7h-3\ee_{i_1}-3\ee_{i_2}-3\ee_{i_3}-3\ee_{i_4}-2\ee_{i_5}-2\ee_{i_6}-2\ee_{i_7}-\ee_{i_8}$&$280$&2\\
7&$7h-4\ee_{i_1}-3\ee_{i_2}-2\ee_{i_3}-2\ee_{i_4}-2\ee_{i_5}-2\ee_{i_6}-2\ee_{i_7}-2\ee_{i_8}$&56&2\\
8&$8h-3\ee_{i_1}-3\ee_{i_2}-3\ee_{i_3}-3\ee_{i_4}-3\ee_{i_5}-3\ee_{i_6}-3\ee_{i_7}-\ee_{i_8}$&8&0\\
8&$8h-4\ee_{i_1}-3\ee_{i_2}-3\ee_{i_3}-3\ee_{i_4}-3\ee_{i_5}-2\ee_{i_6}-2\ee_{i_7}-2\ee_{i_8}$&280&0\\
9&$9h-4\ee_{i_1}-4\ee_{i_2}-3\ee_{i_3}-3\ee_{i_4}-3\ee_{i_5}-3\ee_{i_6}-3\ee_{i_7}-2\ee_{i_8}$&168&2\\
10&$10h-4\ee_{i_1}-4\ee_{i_2}-4\ee_{i_3}-4\ee_{i_4}-3\ee_{i_5}-3\ee_{i_6}-3\ee_{i_7}-3\ee_{i_8}$&70&0\\
11&$11h-4\ee_{i_1}-4\ee_{i_2}-4\ee_{i_3}-4\ee_{i_4}-4\ee_{i_5}-4\ee_{i_6}-4\ee_{i_7}-3\ee_{i_8}$&8&2\
\end{tabular}}}
\end{table}

\begin{proposition}\label{M-count}
If $X_\R=\Rp2\#4\T^2$ then, for any generic $x\in X_\R$,  we have
$$\sum\limits_{A\in \CC_\R^4(X,x)}\cs(A)=112.
$$
\end{proposition}
\begin{proof} Due to Corollary \ref{pencils}(1), $\sum\limits_{A\in \CC_\R^4(X,x)}\cs(A)=\sum\limits_{\beta\in \B_\R^{4}(X)} i^{\q_X(\beta)}$ while according to Table \ref{M-gamma}, $\sum\limits_{\beta\in \B_\R^{4}(X)} i^{\q_X(\beta)}=2[-8+70-168+(280+8)-(280+56)]+420=112.$
\end{proof}

\subsection{Signed count in the case of connected $(M-1)$-surfaces}

\begin{proposition}\label{submaximal-model}
If $X_\R= \Rp2\#3\T^2$, then $X$ admits a real blowup model with 6 real blown up points and code $(-1,-1,\dots,-1)$.
\end{proposition}

\begin{proof}
Like in the proof of Proposition \ref{M-model} we construct a real del Pezzo surface $X$ by blowing up $\P^2$
at three pairs of real points $p_i$ and $p_{i+3}$, $i=1,2,3$, located sufficiently close to each other in each pair.
Then, we additionally blow up at a
a pair of imaginary complex-conjugate points $p_7$ and $p_8$, assuming that the whole configuration of 8 points is generic.
As in the proof of Proposition \ref{M-model},
$\q_X(\ee_i)$ and $\q_X(\ee_{i+3})$
are of opposite signs for each $i=1,2,3$.
This implies $\q_X(h)=1$, and
performing a Cremona transformation based at
those 3 points $p_i$ for which $\q_X(\ee_i)$ is positive,
we obtain a
blowup model with code $(-1,-1,\dots,-1)$, as required.

The same deformation arguments as at the end of the proof of Proposition \ref{M-model} apply and extend the result from the surface $X$ constructed to any real del Pezzo surface of degree 1 with
real locus of the same topological type.
\end{proof}

\begin{table}[h]
\caption{$\B_\R^{4}$-classes  in the case $X_\R= \Rp2\#3\T^2$}\label{submaximal-4-classes}
\scalebox{0.8}{\boxed{\begin{tabular}{c|c|c|c}
bi-level&
type of classes $\beta\in \B_\R^{4}(X)$ &number&$\q_X(\beta)$\\
\hline
1,1&$h-\ee_{i}$&6&2\\
0,4&$2h-\ee_{i_1}-\ee_{i_2}-\ee_{i_3}-\ee_{i_4}$&15&0\\
0,2&$2h-\ee_{i_1}-\ee_{i_2}-\ee_{7}-\ee_{8}$&15&2\\
1,5&$3h-2\ee_{i_1}-\ee_{i_2}-\ee_{i_3}-\ee_{i_4}-\ee_{i_5}-\ee_{i_6}$&6&2\\
1,3&$3h-2\ee_{i_1}-\ee_{i_2}-\ee_{i_3}-\ee_{i_4}-\ee_{7}-\ee_{8}$&60&0\\
0,2&$4h-2\ee_{i_1}-2\ee_{i_2}-2\ee_{i_3}-\ee_{i_4}-\ee_{i_5}-\ee_{7}-\ee_{8}$&60&2\\
0,4&$4h-2\ee_{i_1}-\ee_{i_2}-\ee_{i_3}-\ee_{i_4}-\ee_{i_5}-2\ee_{7}-2\ee_{8}$&30&0\\
0,6&$4h-3\ee_{i_1}-\ee_{i_2}-\ee_{i_3}-\ee_{i_4}-\ee_{i_5}-\ee_{i_6}-\ee_{7}-\ee_{8}$&6&2\\
1,1&$5h-2\ee_{i_1}-2\ee_{i_2}-2\ee_{i_3}-2\ee_{i_4}-\ee_{i_5}-2\ee_{7}-2\ee_{8}$&$30$&2\\
1,3&$5h-3\ee_{i_1}-2\ee_{i_2}-2\ee_{i_3}-2\ee_{i_4}-\ee_{i_5}-\ee_{i_6}-\ee_{7}-\ee_{8}$&$60$&0\\
1,5&$5h-3\ee_{i_1}-2\ee_{i_2}-\ee_{i_3}-\ee_{i_4}-\ee_{i_5}-\ee_{i_6}-2\ee_{7}-2\ee_{8}$&$30$&2\\
0,2&$6h-3\ee_{i_1}-3\ee_{i_2}-2\ee_{i_3}-2\ee_{i_4}-2\ee_{i_5}-2\ee_{i_6}-\ee_{7}-\ee_{8}$&15&2\\
0,4&$6h-3\ee_{i_1}-3\ee_{i_2}-2\ee_{i_3}-2\ee_{i_4}-\ee_{i_5}-\ee_{i_6}-2\ee_{7}-2\ee_{8}$&90&0\\
0,2&$6h-2\ee_{i_1}-2\ee_{i_2}-2\ee_{i_3}-2\ee_{i_4}-\ee_{i_5}-\ee_{i_6}-3\ee_{7}-3\ee_{8}$&15&2\\
1,5&$7h-3\ee_{i_1}-3\ee_{i_2}-3\ee_{i_3}-3\ee_{i_4}-2\ee_{i_5}-\ee_{i_6}-2\ee_{7}-2\ee_{8}$&$30$&2\\
1,3&$7h-3\ee_{i_1}-3\ee_{i_2}-2\ee_{i_3}-2\ee_{i_4}-2\ee_{i_5}-\ee_{i_6}-3\ee_{7}-3\ee_{8}$&$60$&0\\
1,1&$7h-4\ee_{i_1}-3\ee_{i_2}-2\ee_{i_3}-2\ee_{i_4}-2\ee_{i_5}-2\ee_{i_6}-2\ee_{7}-2\ee_{8}$&30&2\\
0,6&$8h-3\ee_{i_1}-3\ee_{i_2}-3\ee_{i_3}-3\ee_{i_4}-3\ee_{i_5}-\ee_{i_6}-3\ee_{7}-3\ee_{8}$&6&2\\
0,4&$8h-4\ee_{i_1}-3\ee_{i_2}-3\ee_{i_3}-3\ee_{i_4}-3\ee_{i_5}-2\ee_{i_6}-2\ee_{7}-2\ee_{8}$&30&0\\
0,2&$8h-4\ee_{i_1}-3\ee_{i_2}-3\ee_{i_3}-2\ee_{i_4}-2\ee_{i_5}-2\ee_{i_6}-3\ee_{7}-3\ee_{8}$&60&2\\
1,3&$9h-4\ee_{i_1}-4\ee_{i_2}-3\ee_{i_3}-3\ee_{i_4}-3\ee_{i_5}-2\ee_{i_6}-3\ee_{7}-3\ee_{8}$&60&0\\
1,5&$9h-3\ee_{i_1}-3\ee_{i_2}-3\ee_{i_3}-3\ee_{i_4}-3\ee_{i_5}-2\ee_{i_6}-4\ee_{7}-4\ee_{8}$&6&2\\
0,2&$10h-4\ee_{i_1}-4\ee_{i_2}-4\ee_{i_3}-4\ee_{i_4}-3\ee_{i_5}-3\ee_{i_6}-3\ee_{7}-3\ee_{8}$&15&2\\
0,4&$10h-4\ee_{i_1}-4\ee_{i_2}-3\ee_{i_3}-3\ee_{i_4}-3\ee_{i_5}-3\ee_{i_6}-4\ee_{7}-4\ee_{8}$&15&0\\
1,1&$11h-4\ee_{i_1}-4\ee_{i_2}-4\ee_{i_3}-4\ee_{i_4}-4\ee_{i_5}-3\ee_{i_6}-4\ee_{7}-4\ee_{8}$&6&2
\end{tabular}}}
\end{table}

If $(X,\conj)$ is an $(M-1)$-surface and $X_\R$ is connected, then
$X_\R= \Rp2\#3\T^2$ and $\W(X)= E_7$
(see
Table \ref{eigenlattices}).
To enumerate the elements of $\B_\R^{4}$ and to determine their
$\q_X$-values, we use the special blowup model given by Proposition \ref{submaximal-model},
which we call a {\it submaximal negative blowup model}.

First of all, we observe that among $2160$ classes
$v\in K^\perp
=E_8$ with  $v^2=-4$
precisely
$\frac{126\times60}{10}=756$ are real, i.e. belong to $R^{4}_\R(X)$.
Same calculation as in $M$-case above shows that
the corresponding 756 elements of $\B^{4}_\R(X)$
split into 11 subsets
listed
in Table \ref{submaximal-4-classes}. In accordance with notation in Proposition \ref{submaximal-model},
by $\ee_i$ (with unspecified value of index) there meant the classes of the 6 real exceptional divisors, while $\ee_7,\ee_8$ specify the pair of complex conjugate imaginary ones.
Furthermore, each type is accompanied by an indication of its {\it bi-level}, that is a pair $(a,b)$ where $a$ is the $\Z/2$-residue of the coefficient at $h$ and
$b$ the number of classes $\ee_1,\dots,\ee_6$ which enter with odd coefficients in the expansion of the element. Due to the choice of the negative blowup model, for $\beta\in \B^{4}(X)$ of bi-level $(a,b)$, the value $\q_X(\beta)$
(shown in Table \ref{submaximal-4-classes}) is equal to $a+b \mod 4$. Therefore, in accordance with Corollary \ref{pencils}(1) and with Table \ref{submaximal-4-classes} the following holds.

\begin{proposition}\label{submax-count}
If $X_\R=\Rp2\#3\T^2$ then, for any generic $x\in X_\R$,  we have
\pushQED{\qed} \[\sum\limits_{A\in \CC_\R^4(X,x)}\cs(A)=\sum\limits_{\beta\in \B_\R^{4}(X^+)} i^{\q_{X^+}(\beta)}=84.
\qedhere
\popQED
\]
\end{proposition}


\section{Proof of Theorems \ref{main-30} and \ref{main}}\label{Proofs}

\subsection{Signed count of curves in $\CC_\R^2(X,x)$}
 \begin{proposition}\label{First-summand}
 For every Bertini pair of real del Pezzo surfaces $X^\pm$ of degree 1,
and any pair of real generic points $x^\pm\in X^\pm_\R$, we have
 $$
 \begin{aligned}
  \sum_{A\in \CC_\R^{2}(X^+,x^+)}   \css(A)=   2(r^--r^+)r^+, & \qquad
 \sum_{A\in \CC_\R^{2}(X^-,x^-)}   \css(A)=    2(r^+-r^-)r^- , \\
 \sum_{A\in
 \CC_\R^{2}(X^+,x^+)\cup \CC_\R^{2}(X^-,x^-)}   \css(A) = & - 2(r^+-r^-)^2.
\end{aligned}
 $$
where $r^\pm=\rk \W(X^\pm)$.
 \end{proposition}

 \begin{proof} Any class $\a= -2K-e\in \B_\R^{2}(X^\pm)$ gives a
 real net of curves of arithmetic genus 1 (see Proposition \ref{rulings-prop}).
By fixing a generic basepoint $x\in X_\R^\pm$, we obtain a
real pencil, whose other basepoint $x'\ne x$ has to be real (see Corollary \ref{pencils}).
Singular curves from this pencil are irreducible (and thus, rational, with one node), except one curve,
which has to be real and gives a splitting $\a=(-K)+(-K-e)$.
The first, {\it anticanonical}, component is
a real curve of genus 1,
passing through $x$, and the second component is a real line.
After blowing up the points $x,x'$ we obtain
a real fibration $X_\R^\pm\#2\Rp2\to\Rp1$, so that counting the Euler characteristic of $X_\R^\pm\#2\Rp2$
by means of this fibration we get the relation (the last equality is due to the Lefschetz fixed-point formula)
\begin{equation}\label{euler-integration}
 -1+ \sum_{A\in  \CC_\R^{2}(X,\a,x^\pm)}w(A)=\chi(X_\R^\pm)-2=r^{\mp}-r^{\pm}-1
\end{equation}
where $-1$ is the Euler characteristic of the reducible fiber and
$w(A)$ stands for
$\chi(A_\R)=s_A-c_A$
(where $s_A$ and $c_A$ stand for the number of solitary and cross-point nodes, respectively).

On the other hand, $\sum\limits_{\alpha\in \B_\R^{2}(X^\pm)} i^{\q_{X\pm}(\a)}=2r^\pm$
due to Proposition 3.4.5 in \cite{TwoKinds}
(where it is written in a bit different notation, so that $\rk \Lambda(X)$ becomes $\rk R_\R(X)$ and  $\sum\limits_{\alpha\in \B_\R^{2}(X)} i^{\q_{X}(\a)}$
turns into $\sum\limits_{l\in L_\R(X)} s(l)$).
Since furthermore $\css(A)=i^{\q_{X\pm}(\a)}w(A)$, we conclude that
$$
\sum_{A\in \CC_\R^{2}(X,x^\pm)}   \css(A)=
\sum_{\alpha\in \B_\R^{2}(X^\pm)}i^{\q_{X\pm}(\a)} \sum_{A\in  \CC_\R^{2}(X,\a,x^\pm)}w(A)=2r^{\pm}(r^{\mp}-r^{\pm}).
$$

The third identity in the statement is nothing but the sum of the first two.
 \end{proof}

 \subsection{On deformation invariance of partial counts}\label{partial-invariance}
 As it could be already observed in the proof of Proposition \ref{First-summand}, the value of a sum
\begin{equation}\label{sum}
 \sum_{A\in  \CC_\R^{2}(X,\a,x^\pm)}i^{\q_{X\pm}(\a)} w(A)
\end{equation}
does not depend on a choice of the point $x$, and moreover is invariant under real deformations of $X^\pm$.
The same property holds for  $\sum_{A\in  \CC_\R^{2k}(X,\a,x^\pm)}i^{\q_{X\pm}(\a)} w(A)$ with $k=0$ and $2$.
In fact, such an invariance property holds in much more general situation.

 \begin{proposition}\label{universal-invariance}
For any $\alpha\in H_2(X)$, a function $f:\Z\times\Z\times  \Z/4\to \R$,
and
a $\conj $-invariant collection $\bf x\subset X$ of $\Card(\bf x)=-\alpha\cdot K_X -1$ points,
the sum
$$\sum_A f([\alpha]^2,\alpha\cdot K_X, \q_X(\alpha))w(A)$$
taken over all real rational curves $A$ with $[A]=\alpha$, ${\bf x}\subset A$,
depends only on $\alpha$, $f$ and the number $\Card({\bf x}\cap X_\R)$
of real points, but not on the set $\bf x$ itself.
Such a sum is also invariant under real deformations of $X$.
 \end{proposition}

 \begin{proof} Due to real deformation invariance of $q_X$ (see Theorem \ref{main-Pin}), the stated invariance of the sum (\ref{sum})
 under real deformations of $X$ and a choice of the point-collection $\bf x\in X$ follows from the same kind of invariance of the modified Welschinger invariant
 $\sum_{A}w(A)$. The latter one coincides with the genuine Welschinger invariant up to the sign $(-1)^{g_a}$ where $g_a=\frac{\a^2+\a K}2+1$ is the arithmetic
 genus of the curves $A$ with $[A]=\a$. For the genuine Welschinger invariant,  the invariance in question is established in \cite{Br2}.
 \end{proof}


\subsection{Signed count of curves in $\CC_\R^4(X,x)$
if $X_\R$ is other than $\Rp2\#4\T^2$ and $\Rp2\#3\T^2$}
In the case-by-case analysis
of each of the lattices $\W(X^\pm)$
from  Table \ref{eigenlattices} we use such a root basis on which $\q_X$ is vanishing identically. For existence of such a basis, see \cite[Lemma 3.1.2]{TwoKinds}.

\subsubsection{\bf Cases  $X_\R=\Rp2 \dsum k\S^2$}\label{case4A}
If $X_\R=\Rp2 \dsum k\S^2$, $0\le k\le4$, then $\W(X)$ is
isomorphic to $(4-k)A_1$.
Since each $(-4)$-vector in $(4-k)A_1$
splits into a sum of generators of a pair of $A_1$-summands, the number of such vectors is $4\binom{4-k}2$ which is
$0$ for
$k$ equal 3 and 4.
Each $(-4)$-vector $v$ yields a unique curve $A\in  \CC_\R^4(X,-2K_X-v, x)$ (see Corollary \ref{pencils}(1))
and
for them
the value
$\css(A)=i^{\q_{X}(v)}$ is
$1$ since $\q_X=0$ on each $A_1$-summand. Thus, we just count the number
$\operatorname{Card}(R_\R^{4})$ of (-4)-vectors in $\W(X)$:
$$
\sum_{A\in  \CC_\R^4(X,x)}\css(A)=
\operatorname{Card}(R_\R^{4})=4\binom{4-k}2.
$$

\subsubsection{\bf Cases $X_\R=\Rp2\dsum\K^2$ and $X_\R=(\Rp2\#\T^2)\dsum \S^2$}\label{casesD4}
According to Table \ref{eigenlattices}, if $X^+_\R=\Rp2\dsum\K^2$
(resp. $X^+_\R=(\Rp2\#\T^2)\dsum \S^2$) then
$X^-_\R=\Rp2\dsum\K^2$ (resp. $X^-_\R=(\Rp2\#\T^2)\dsum \S^2$) too, and in all the cases
both $\W(X^\pm)$ are isomorphic to $D_4$.
Note that $D_4$ can be seen as
a sublattice of $4 \langle -1\rangle$ generated by the roots
 $e_0=(1,1,0,0), e_1=(1,-1,0,0), e_2=(0, 1,-1,0)$ and $e_3= (0,0,1,-1)$.
With respect to this presentation, the (-4)-vectors split into 2 kinds: sixteen vectors $(\pm 1, \pm 1, \pm 1, \pm 1)$ and eight vectors with one coordinate $\pm 2$ and the others $0$. Thus, the vanishing of $\q_X$ on $e_0,\dots, e_3$ implies its vanishing on all the (-4)-vectors in $D_4$. Indeed, it has then to vanish on
 $e_1+e_3$ and thus on all the sixteen first-kind (-4)-vectors (as they are congruent modulo $2D_4$ to each other), while vanishing
 of $\q_X$ on the eight second-kind (-4)-vectors follows from their presentation as
a sum of $(\operatorname{mod}\,2)$-orthogonal first-kind vectors.
For each of $X=X^\pm$, this gives (in accordance with Corollary \ref{pencils}(1))
$$
\sum_{A\in  \CC_\R^4(X,x)}\css(A)=
\sum_{v\in R_\R^{4}} i^{\q_{X^\pm}(v)}=
16+8=24.
$$

\subsubsection{\bf The case of $X_\R=\Rp2\#\T^2$ }
According to Table \ref{eigenlattices}, $\W(X)$  is $D_4\oplus A_1$.
A (-4)-vector in $D_4\oplus A_1$ is either one of the (-4)-vectors of $D_4$ (24 choices), or a sum of one root in $D_4$ (24 choices) with one root in $A_1$ (2 choices).
On the (-4)-vectors of the first-kind the form
$\q_X$ vanishes, like in the previous case.
On the latter sums we have  $\q_X(v_1+v_2)=\q_X(v_1)$, and in accordance with  \cite{TwoKinds} the signed count of the 2-roots $v_1$ gives $2 \rk|D_4|=8$,
 which after that should be multiplied by 2 because of two choices of $v_2$ in $A_1$.
Thus, in accordance with Corollary \ref{pencils}(1),
$$
\sum_{A\in  \CC_\R^4(X,x)}\css(A)=
\sum_{v\in R_\R^{4}} i^{\q_{X}(v)}=
24+16=40.
$$

\subsubsection{\bf The case of $X_\R=\Rp2\#2\T^2$}
According to Table \ref{eigenlattices}, 
$\W(X)$  is isomorphic to $D_6$, which
can be seen as
a sublattice of $6\la -1\ra$ generated by the following roots
\scalebox{0.7}{$\boxed{\begin{matrix}
-1&-1&0&0&0&0\\
1&-1&0&0&0&0\\
0&1&-1&0&0&0\\
0&0&1&-1&0&0\\
0&0&0&1&-1&0\\
0&0&0&0&1&-1\\
\end{matrix}}$}
The vanishing of $\q_{X}$ on these basic roots implies immediately its vanishing
on the twelve (-4)-vectors
that contain $\pm2$ as one coordinate and $0$ as others.
The rest of
$(-4)$-vectors are obtained from $(\pm1,\pm1,\pm1,\pm1,0,0)$ by permutation of coordinates (totally $2^4\binom62$).
 The $\binom62$ permutations interpreted as partitions $n_1+n_2+n_3=4$ have either all $n_i$ even ($6$ possibilities),
 in which case $\q_X=0$,
 of give two odd summands $n_i$ ($9$ cases), in which case $\q_X=2$.

Thus, in accordance with Corollary \ref{pencils}(1),
$$
\sum_{A\in  \CC_\R^4(X,x)}\css(A)=
\sum_{v\in R_\R^{4}} i^{\q_{X}(v)}=
12+16(9-6)=60.
$$

\subsection{Proof of Theorems \ref{main} and \ref{main-30}}
The results obtained above are summarized in Table \ref{numbers}
which is organized by columns according to Smith types of Bertini pairs
and where the rows show the result of the signed count
over $A\in\CC_\R^{2k}(X,x)$ for $X=X^\pm$ in each Bertini pair. For the first 4 columns our convention is that
$X^+$ refers to surfaces with connected real locus (see Table \ref{eigenlattices}).

For $k=1$, the results shown in Table \ref{eigenlattices}
are given by Proposition \ref{First-summand}.
For $k=2$, they
are taken from Propositions \ref{M-count}, \ref{submax-count} and
previous Subsection. For $k=0$
they
are due to  \cite{Br-Floor} (note that the original Welschinger weight used in \cite{Br-Floor} coincides with our $\css$-weight in the case of curves $A\in\CC_\R^0(X,x)$, since their
arithmetic genus $g_A=2$ is even, and thus $c_A$ and $s_A$ have the same pairity).

Adding the 3 terms
$\sum_{A\in  \CC_\R^{0}(X,x)}\css(A)+\sum_{A\in  \CC_\R^{2}(X,x)}\css(A)+\sum_{A\in  \CC_\R^{4}(X,x)}\css(A)$
for each type of $X=X^\pm$ we obtain Theorem \ref{main-30}.

As we take
the sum $\sum_{A\in  \CC_\R^{2}(X,x)}\css(A)+2\sum_{A\in  \CC_\R^{4}(X,x)}\css(A)$ for $X=X^+$ and
add it with the same sum for $X=X^-$, we obtain Theorem \ref{main}.

\begin{table}[h]
\caption{$\sum\limits_{A\in  \CC_\R^{n}}\css(A)$ for each type of Bertini pairs $X^\pm$
}\label{numbers}
\hskip-8mm
\hbox{\vtop{\hsize10.7cm\boxed{\begin{tabular}{c|c|c|c|c|c|c}
&$M$&$M\!-\!1$&$M\!-\!2$&$M\!-\!3$&$M\!-\!4$&$(M\!-\!2)_I$\\
\hline
$A\in\CC_\R^{2}(X^+,x^+)$&-128&-84&-48&-20&0&0\\
$A\in\CC_\R^{2}(X^-,x^-) $&0&12&16&12&0&0\\
\hline
$A\in\CC_\R^{4}({X^+},x^+)$&112&84&60&40&24&24\\
$A\in\CC_\R^{4}({X^-},x^-)$&0&0&4&12&24&24\\
\hline
$A\in\CC_\R^{0}({X^+},x^+)$&46&30&18&10&6&6\\
$A\in\CC_\R^{0}({X^-},x^-)$&30&18&10&6&6&6\\
\end{tabular}}}
\vtop{\vskip-9mm
\hbox{$4r(4-r)$}\vskip4mm
\hbox{$2r(r-1)$}\vskip4mm
\hbox{\hskip-1mm$2(r-3)(r-4)+6$}
}}
\vskip0.05in 
To the right from the table, for each $n=0,2,4$, these values are expressed as a single function of $r={\rm rk}\W(X^\pm)$.
\end{table}
\vskip-5mm

\section{Wall crossing}\label{wall-crossing}
\setlength\epigraphwidth{.63\textwidth}
\epigraph{Richtiges Auffassen einer Sache und  Mißverstehen der  \\
gleichen Sache schließen einander nicht vollständig aus.}{F.~Kafka, "Der Prozeß. Kapitel 9:  Im Dom" \\}

According to
Theorems \ref{main-30} and \ref{main} the weighted counts from these theorems are invariant under choice of a del Pezzo surface and choice of a generic point. The proof presented above (in Sections \ref{max-submax} - \ref{Proofs}) is achieved by a somehow case-by-case calculation.
Here, we provide a somewhat
 direct proof of such an invariance
based on a real version of Abramovich-Bertram-Vakil wall-crossing formula \cite[Theorem 4.2]{Vakil}
and the underlying
gluing procedure, see \cite[Proposition 4.1]{IKS} and \cite[Theorem 2.5]{Br-P}.

\subsection{Wall-crossing families}\label{families-setting}
The bi-anticanonical map establishes an isomorphism
between the moduli space of Bertini pairs of
real del Pezzo surfaces $X$ of degree 1 and that of
real non-singular sextics $C$ on the real quadric cone $Q$  based on a non-empty real conic. In particular, such an isomorphism allows not only to identify the real deformation classes of these objects
but also to visualize nodal degenerations of the former with nodal degenerations of the latter.

Let us say that a germ of a complex analytic family
of sextics $C(t)\subset Q$, $t\in D^2=\{t\in\C\,|\, \vert t\vert <1\}$ form
a (simple) {\it nodal degeneration}, if $C(t)$ are non-singular sextics for $t\ne 0$ while $C(0)$ is uninodal.
We name
a nodal degeneration a {\it Morse-Lefschetz family} if the total space of the family is smooth
(in other words, ``wall-intersection'' is transverse),
and {\it real} if
the complex conjugation on $Q\subset\P^3$
maps $C(t)$ to $C(\bar t)$.

By taking the double coverings $X(t)\to Q$ branched along $C(t)$ we obtain a complex analytic
family $\{X(t)\}_{t\in D^2}$ where $X(t)$ are
 del Pezzo surfaces of degree 1 for $t\ne 0$,
while $X(0)$ is uninodal. Such a family will be called a {\it Morse-Lefschetz
family of del Pezzo surfaces of degree 1},
if the family $C(t)$ is Morse-Lefschetz.

If the family $C(t)$ is real, then the real structure $\conj: Q\to Q$ lifts to a
pair of Bertini-dual real structures on the total space of the family $\{X(t)\}_{t\in D^2}$.
 In particular, for each $t\in\R\cap D^2$, the
 surface $X(t)$ acquires a pair of Bertini dual real structures, whose real loci are denoted $X_\R^\pm(t)$.

A real Morse-Lefschetz family $X(t)$ will be called {\it bi-cone-like} if the real
node on each of $X_\R^\pm(0)$ is modeled locally
by equation $x^2+y^2-z^2=0$ (rather than by $x^2+y^2+z^2=0$).
In terms of the exceptional divisor $E\subset\til X(0)$ of the blowing up $\til X(0)\to X(0)$ this means that the real locus
$E^\pm_\R\subset \til X^\pm_\R(0)$ is non-empty for both Bertini-dual real structures lifted to $\til X(0)$.
In terms of $C(t)$, this means that the node of $C_\R(0)$ is a cross-point (rather than a solitary point).

\begin{proposition}\label{special-wall} Any
two
real del Pezzo surfaces of degree 1 can be connected by a finite sequence of real
Morse-Lefschetz bi-cone-like
nodal degenerations.
\end{proposition}

\begin{proof}
Let $C\subset Q$ be a real non-singular sextic whose real locus is connected.
Then, $X^\pm_\R$ is homeomorphic to $\Rp2$ for both del Pezzo surfaces $X^\pm$ of the Bertini pair defined by $C$.
Therefore, these $X^\pm$ are real deformation equivalent to each other. Thus, there remain to check
that any real non-singular sextic $C'$ can be obtained from $C$ by a finite sequence of real
Morse-Lefschetz families avoiding solitary nodes. Such degenerations can be found, for example, on Fig.1 in \cite{TwoKinds}.
\end{proof}

In the rest of Section \ref{wall-crossing} we always suppose that:
\begin{enumerate}\item
$\{X(t)\}_{t\in D^2}$ is a real Morse-Lefschetz bi-cone-like family
of del Pezzo surfaces of degree 1;
\item
$\chi(X_\R(t))<\chi(X_\R(-t))$ for $t>0$ (for a bi-cone-like family this means that the node of $X_\R(0)$ gives one-sheeted perturbation on $X_\R(t)$
with $t>0$
and two-sheeted on $X_\R(-t)$);
\item
$x=\{x(t)\}_{t\in D^2}$ is
a $\conj$-invariant family
of basepoints, $x(t)\in X(t)$,
in particular, for $t\in\R$,
$x(t)\in X_\R(t)$;
\item
the family $x$ is {\it generic}, and
$x(0)$ is chosen on
a real arc issued from the node in a generic direction, not at the node but
sufficiently near to it.
Such choice
(used, in particular,  in \ref{conic-bundle-case})
is possible due to our assumption that the node is bi-cone-like.
\end{enumerate}

\subsection{Merging families of curves}\label{merging-families}
For each
$n\in \{0,2,4\}$ and a sufficiently small $\e>0$, the family
$$
\cup_{t\in[-\e,\e], t\ne 0}\CC^{n}_\R(X(t),x(t))
$$
has a natural structure of semi-analytic set with a proper, having finite fibers, projection to
$[-\e,\e]\sm 0$. Their union has a natural compactification
by a finite set $\CC_\R(X(0),x(0))$
formed by the curves on $X(0)$ obtained as limits of the curves from
$\cup_{t\in[-\e,\e], t\ne 0}\,\CC^{n}_\R(X(t),x(t))$. This gives us
what we call a {\it wall-crossing diagram}: a compact 1-dimensional graph-like set
$$
\CC_\R(X,x)=\cup_{t\in[-\e,\e]}\CC_\R(X(t),x(t)),
$$
with a proper map $\pi:\CC_\R(X,x)\to
[-\e,\e]$
that sends $\CC^{n}_\R(X(t),x(t))$ to $t$ ({\it cf.}, \cite[Proposition 4.1]{IKS}).
We are interested only in the germ of this map at
$\CC_\R(X(0),x(0))$ and treat each of the elements of the latter as vertices, although some of them represent a smooth point with a non-singular projection,
in which case we call a vertex {\it inessential}.
More precisely, it happens if the curve $A$ representing a vertex is not passing through the node of $X(0)$; then it extends to a unique smooth family
$A(t)\subset X(t)$, $A(0)=A$, $x(t)\in A(t)$, $t\in[-\epsilon,\epsilon]$.

For a fixed
real
Morse-Lefschetz family $(X(t),x(t))$, there exists $\epsilon>0$ such that each {\it positive} (respectively, {\it negative}) {\it semi-branch} issued from any vertex $A\in \CC_\R(X(0),x(0))$
represents a unique smooth family $A(t)\subset X(t)$, $A(0)=A$, $x(t)\in A(t)$, $t\in[0,\epsilon]$ (respectively, $t\in[-\epsilon, 0]$).
We name such germs of semi-branches {\it positive and negative edges}.
A vertex has {\it branch signature} $(p,q)$ if it has degree $m=p+q$ with $p$ positive and $q$ negative incident edges.
The set of all vertices of a fixed signature $(p,q)$ is denoted by $\V_{p,q}$ and the set of all vertices by $\V=\cup_{p,q\ge0}\V_{p,q}$.

The set of all edges
of the wall-crossing diagram
has a natural
partition according to a partition of the
sets  $\B^n(X(t))$, $t\ne0$, into
 $$\B^{n,k}(X(t))=\{
\alpha\in B^n(X(t))\,|\,
\alpha\cdot v=\pm k\},\ \  k\ge0,$$
depending on the intersection index $x\cdot v$
with the vanishing class $v\in H_2(X(t))$ of the node of $X(0)$. Namely,
for $n=2, 4$ and every $t\ne 0$,
replacing divisor classes $\alpha\in B^n(X(t))$ by curves $A$ representing $\alpha$,
we obtain a partition
\begin{equation*}
\CC_\R^{n}({X(t)},x(t))= \CC_\R^{n,0}(X(t),x(t))\cup\CC_\R^{n,1}(X(t),x(t))\cup\CC_\R^{n,2}(X(t),x(t)),
\end{equation*}
where
\begin{equation*}
\CC_\R^{n,k}(X(t),x(t))= \{A\in \CC^{n}_\R(X(t),x(t))\,|\, [A]\cdot v=\pm k\},\ k\ge0
\end{equation*}
(for $n=0$,  we get
a trivial splitting
$\CC^0_\R(X(t),x(t))=\CC^{0,0}_\R(X(t),x(t))$,
since $\B^0(X)$ contains only one element, $-2K_X$).
Such a partition
is preserved under deformations and induces a partition of the
edges into (n,k)-{\it types}. We denote by
$\E^{n,k}_+(v)$ (resp. $\E^{n,k}_-(v)$) the set of positive (resp. negative) edges of the given $(n,k)$-type
incident to a given vertex $v\in\V $,
and let $\E_\pm(v)=\cup_{n,k}\E^{n,k}_\pm(v)$.

The value $\css(A(t))$ does not change with $t$ for the curves $A(t)$ representing an edge $\e\in\E$, so we get a well-defined sign $\css(e)=\css(A(t))$.
To measure how the signed count of edges is effected by a wall-crossing, we define
$$
\Delta^{n,k}_{p,q}=\sum_{v \in \V_{(p,q)}} \left(\sum_{\e \in\E^{n,k}_+(v)} \css(\e)- \sum_{\e \in\E^{n,k}_-(v)}\css(\e)\right).
$$

\begin{proposition}\label{merging}
Assume that $(X(t),x(t))$
is a pointed Morse-Lefschetz family
satisfying assumptions
(1)-(4)
of Subsection \ref{families-setting}.
Then each vertex of its wall-crossing diagram is either
inessential, or has branch signature $(2,0)$, $(4,0)$, or $(2,2)$.
Moreover:
\begin{itemize}
\item[(A)]
At each vertex $v\in \V$ we have
$\sum_{e\in\E_+(v)}=\sum_{e\in\E_-(v)}$,
and both sums vanish for $v\in\E_{(2,0)}\cup\E_{(4,0)}$.
\item[(B)]
The values of $\Delta^{n,k}_{p,q}$ are
as indicated in Table \ref{wall-crossing-table}
\end{itemize}
\end{proposition}

This Proposition will be proved in Subsection \ref{merging}. Before
we derive from it the invariance property of the counts from Theorems \ref{main-30} and \ref{main},
and make preparations in Sections \ref{enumerating-limits} and \ref{splittings}.

\begin{table}[h]
\caption{$\Delta^{n,k}_{p,q}$ for inessential vertices
($r^\pm$ denotes $\rk K^\perp\cap\ker(1\pm \conj_*)$ taken for $X(t)$ with $t>0$) }\label{wall-crossing-table}
\boxed{
\begin{tabular}{c|c|c}
$(n,k)-type$&$(p,q)$-signature&$\Delta^{n,k}_{p,q}$\\
\hline
$(4,1)$&(2,2)&$0$\\
$(4,2)$&(4,0)&$4(r^+-1)$\\
\hline
$(2,0)$&(4,0)&$-4(r^+-1)$\\
$(2,1)$&(2,0)&$0$\\
$(2,2)$&(4,0) or (2,2)&$-2(r^+-r^-)$\\
\end{tabular}}
\end{table}

\subsection{Derivation of
Theorem \ref{main-30}}
Due to Proposition \ref{universal-invariance},
our count does not change under real variations of $X$ as
long as $X$
is preserved non-singular,
and under variations of $x\in X_\R$.
Thus, for proving Theorem \ref{main-30},
it is sufficient to check the invariance under the {\it wall-crossing},
{\it i.e.,} in real Morse-Lefschetz families of real del Pezzo surfaces of degree 1, $X(t), t\in\C, \vert t\vert \!<\!\!<\!1$
with $\conj(t): X(t)\to X(\bar t)$.
Due to Proposition \ref{special-wall}, we may
restrict ourselves only with nodal degenerations whose exceptional divisor $E$ has
$E_\R\ne\varnothing$ both for $X^+(0)$ and $X^-(0)$. Then,
Proposition \ref{merging}(A) applies and shows even a stronger invariance statement: the invariance of the count
(from Theorem \ref{main-30}) at each vertex of the wall-crossing diagram. 

To obtain the value of $\sum_{A\in  \CC_\R(X,x)}\css(A)$, it is enough now to consider 
a surface $X$ with $X_\R=\Rp2\+4\S^2$. Then,  $\Lambda(X)=0$ (see Table \ref{eigenlattices}) which implies $\B^2(X)=\B^4(X)=\varnothing$.
Thus the above sum is reduced to the Welschinger invariant of $-2K$, which is $30$ due to \cite{Br-Floor}.\qed

\subsection{Derivation of
Theorem \ref{main}}
As above, due to Propositions \ref{universal-invariance} and \ref{special-wall} we  may
restrict ourselves only with nodal degenerations whose exceptional divisor $E$ has
$E_\R\ne\varnothing$ both for $X^+(0)$ and $X^-(0)$. Then, the
invariance claimed
follows from weighted summation for classes $\CC^2_\R$ and $\CC^4_\R$ indicated in the rightmost column of Table \ref{wall-crossing-table}.
Namely, it gives
$8(r^+-1)-4(r^+-1)-2(r^+-r^-)=2(r^++r^-)-4=12$ coming from $X^+$ and the opposite sum from its Bertini dual,
$-2(r^++r^-)+4=-12$. Alternation of the sign here is due to
our convention $\chi(X_\R^+(t))>\chi(X_\R^+(-t))$ for $t>0$, which implies that $\chi(X_\R^-(-t))>\chi(X_\R^-(t))$,
since $\chi(X^+_\R(t))+ \chi( X^-_\R(t))=2$.

To obtain the value of
$$
\sum_{A\in  \CC_\R^2(X^+,x^+)\cup\CC_\R^4(X^+,x^+)}\cs(A) +\sum_{A\in  \CC_\R^2(X^-,x^-)\cup\CC_\R^4(X^-,x^-)}\cs(A),
$$
we consider a Bertini pair $X^\pm$ where $X^+_\R=\Rp2\#4\T^2$ and $X^-_\R=\Rp2\+4\S^2$. The second sum is equal to $0$, since
it is empty because of  $\B^2(X^-)=\B^4(X^-)=\varnothing$ ({\it cf.} the proof above). To treat the first sum, note that by definition of $\cs$ this sum is equal to
$\sum_{A\in  \CC_\R^2(X^+,x^+)} \s(A) + 2 \sum_{A\in  \CC_\R^4(X^+,x^+)} \s(A)$. Here, according to Proposition \ref{First-summand}, 
$\sum_{A\in  \CC_\R^2(X^+,x^+)} \s(A)= 2(r^--r^+)r^+= - 128$. By Theorem \ref{main-30} we also have 
$\sum_{A\in \R^0(X^+,x^+)} \s(A) + \sum_{A\in  \CC_\R^2(X^+,x^+)} \s(A) +  \sum_{A\in  \CC_\R^4(X^+,x^+)} \s(A) = 30$ while $\CC_{A\in \R^0(X^+,x^+)} \s(A)=46$
(see Table \ref{numbers}). Taking this into account we get $-128 + 2(30-46+128)=96$.
\qed
\vskip0.1in

\subsection{Untwisting of Morse-Lefschetz families}
\label{enumerating-limits}
To kill the monodromy in a
Morse-Lefschetz family $X(t), t\in\C, \vert t\vert <\!\!<1$, as above,
we consider
an {\it associated untwisted family}, $X'(\tau), \tau\in\C, \vert \tau\vert <\!\!<1$, induced
by the base change $t=\tau^2$.
The total space,
3-fold $X'=\cup_{\tau} X'(\tau)$,
acquires a node at the nodal point of $X'(0)=X(0)$.
Blowing up at this node, $\til X\to X'$,
leads to a new
a family $\til X(\tau),  \tau\in\C, \vert \tau\vert <\!\!<1$ with $\til X(\tau)=X'(\tau)$ for $\tau\ne 0$
and $\til X(0)$ formed by two irreducible components with normal crossing: one component,
denoted $\til X^1(0)$, is the minimal nonsingular model of $X'(0)$, and the other one, denoted $\til X^0(0)$,
is isomorphic to $\P^1\times \P^1$.
These components
intersect each other along a nonsingular rational curve $E$ that represents the exceptional divisor of the blowup
$\til X^1(0)\to X'(0)$ and which
can be seen in $\til X^0(0)=\P^1\times \P^1$ as the diagonal.

By contraction $\til X^0(0)\to E$ along the lines
of one of the rulings,
we get a smooth family
of smooth surfaces with $\til X^1(0)$ as central fiber
(see \cite{A} for details). A choice of such a contraction provides then
 natural isomorphisms
\begin{equation*}
\Pic(\tilde X(\tau))=H_2(\til X(\tau)) \overset{\simeq}{\longrightarrow}H_2(\til X^1(0))=\Pic(\til X^1(0))
\end{equation*}
that
preserve the intersection form.
This allows us for simplicity of notation to use
the same symbol for all corresponding divisors and homology classes.
Although we also have a natural isomorphism between $H_2(\til X(\tau))$ and $H_2({X}({t}))$ with $t=\tau^2$ for $\tau\ne 0$,
the composed map
$$H_2({X}({t}))\overset{\simeq}{\longleftarrow} H_2(\til X(\tau)) \overset{\simeq}{\longrightarrow}H_2(\til X^1(0))
\overset{\simeq}{\longleftarrow} H_2(\til X({-\tau})) \overset{\simeq}{\longrightarrow}H_2({X}({t}))$$
is not identity but the Dehn twist $x\mapsto x+([E]\circ x)[E]$. This does not allow us to transport such a simplification of notation to the whole family
$X(t)$ with $t\ne 0$. However, if the Morse-Lefschetz family is equipped with a real structure, then there appear two lifts of the real structure to $\tilde X$: one of them is coherent with the standard complex conjugation
$\tau\mapsto \bar \tau$, and the other one with $\tau\mapsto -\bar\tau$. As a rule we privilege the first one, which, for real $t>0$, identifies $X(t)$ as a real surface with $\tilde X(\tau)$
where $\tau>0$, $\tau^2=t$.
When we turn to $X(t)$ with $t<0$, we use the second lift and identify $X(t)$ with $\tilde X(\tau)$, $\tau^2=t$, where $\tau$ is pure imaginary with positive imaginary part.
In such a way we transport the above simplification of notation to $X(t)$ with $t<0$ too.
This allows us also to simplify
our notation $\B(X)=\B^0(X)\cup\B^2(X)\cup\B^4(X)$ when applied to
$X=\til X(\tau)$ with $\tau\ne 0$,  to $X=\til X^1(0)$, and
 to $X=X(t)$ with any real $t\ne 0$.

 Note also that the real structures on $\tilde X(\tau)$, $\tau >0$, and on $\tilde X(i \tau), \tau>0,$ converge to the same real structure on $\tilde X^1(0)$
 and to different real structures on $\tilde X^0(0)=\P^1\times \P^1$. Due to assumption (2) in Subsection \ref{families-setting}, the first one acts on
 $\tilde X^0(0)=\P^1\times \P^1$ as
 $(p,q)\mapsto (\bar p, \bar q)$, while the second one
 as $(p,q)\mapsto (\bar q, \bar p)$.

\subsection{Enumerating of limit splittings}\label{splittings}
In this subsection we consider a Morse-Lefschetz family $(X(t),x(t))$ as in Proposition \ref{merging}.
As is known (see \cite[Theorem 4.2]{Vakil} and \cite[Proposition 4.1]{IKS}), each family
$A(t)\in \CC_\R^{2}(X(t),x(t))\cup\, \CC_\R^{4}( X(t),x(t))$, $0\!<\! t\! <\!\epsilon$ (resp., $-\epsilon\! <\!t\!<\!0$), in $X(t)$ after a canonical lift
$\til A(\tau): =A(\tau^2)$ to $\til X(\tau)$ with $0\!<\!\tau\!<\!\sqrt{\epsilon}$ (resp., with $\tau=i \tau', 0\!<\!\tau'\!<\!\sqrt{\epsilon}$) admits a well defined, and unique, extension up to a family $\til A(\tau)$ defined for every
$\tau\in \C, \vert\tau\vert <\sqrt{\epsilon}$.
As the first step of the
proof of Theorem \ref{main}, we enumerate
possible splittings $\til A(0)=D{}+ rE$, $r>0$,
in the limit $\til A(0)=\lim_{\tau\to 0} \til A(\tau)\subset \til X^1(0)$. As is known ({\it Loc.cit.}), in each of these splittings, due to a generic choice of $x(t)$, the divisors $D$ are always
reduced irreducible rational and intersect $E$ transversally.

\begin{lemma}\label{e-walls}
If $A(t)\in \Cal C^{2}(X(t),x(t))$, so that
$[A(t)]= -2K-e\in \Cal B^2$,
then the only 3 cases of splitting $\til A(0)=D{}+ rE$, $r>0$, are
like indicated in the table below. In the third case, we have $e=-[E]$.
\newline
\centerline{\boxed{
\begin{tabular}{c|c|c|c|c|c|c}
case&r&$e\cdot [E]$&$[D]$& class of $D$&$[D{}]^2$&$[D{}]\cdot [E]$\\
\hline
$(1)$&$1$&$1$&$-2K-e-[E]$&$\B^2$&$2$&$1$\\
$(2)$&$1$&$0$&$-2K-e-[E]$&$\B^4$&$0$&$2$\\
$(3)$&$2$&$2$&$-2K-[E]$&$\B^2$&$2$&$2$
\end{tabular}}}
\end{lemma}

\begin{proof} The divisor $D$, which is irreducible and pass, by definition, through a fixed generic point, 
can not be a component of the proper image $I\subset \til X^1(0)$ of the generator of $Q$ passing through the node of $C_0$.
Since $[D]=[\til A(0)]-r[E]=-2K-e-r[E]$ and $[I]=-K-[E]$, 
we conclude that $[D]\cdot(-K-[E])=2+e\cdot [E]- 2r\ge0$. This gives
 bound $r\le 2$, since  $e\cdot [E]\le 2$ by Cauchy-Schwarz inequality.
 Furthermore, by the same argument, if $r=2$ then
 $e\cdot [E]= 2$
 and thus $[E]=-e$. Both $r=1$ and $e\cdot [E]=2$ is impossible,
since then $e=-[E]$ and  $[D{}]=-2K$ which contradicts to $[D{}]\cdot [E]\ge 1$.
The remaining calculations trivially follow from $[D]=-2K-e-r[E]$.
\end{proof}

\begin{lemma}\label{v-walls}
If $A(t)\in \Cal C^{4}(X(t),x(t))$,
so that
$[A(t)]= -2K-v\in \Cal B^4$,
then the only 2 cases of splitting $\til A(0)=D{}+ rE$, $r>0$,
are like indicated in the table below
\newline
\centerline{
\boxed{
\begin{tabular}{c|c|c|c|c|c|c}
case&r&$v\cdot [E]$&$[D]$& class of $D$&$[D{}]^2$&$[D{}]\cdot [E]$\\
\hline
$(1)$&$1$&$1$&$-2K-v-[E]$&$\B^4$&$0$&$1$\\
$(2)$&$2$&$2$&$-2K-v-2[E]$&$\B^4$&$0$&$2$\\
\end{tabular}}}
\end{lemma}

\begin{proof}
Since $v\cdot [E]\le [2\sqrt2]=2$, the same arguments as in the proof of Proposition \ref{e-walls} imply $r\le 2$ and show that $r=2$ may hold only if $v\cdot [E]=2$.
The case $r=1,  v\cdot[E]=2$ is excluded using $[D]\cdot [E]\ge 1$, and the case $r=1, v\cdot [E]=0$ using $[D]^2\ge 0$.
The remaining calculations trivially follow from $[D]=-2K-v-r[E]$.
\end{proof}

\begin{lemma}\label{0-walls}
If $A(t)\in \Cal C^{0}(X(t),x(t))$, so that
$[A(t)]= -2K$,
then the only splitting $\til A(0)=D{}+ rE$, $r>0$, is
\newline
\centerline{
\begin{tabular}{|c|c|c|c|c|}
\hline
r&$[D]$& class of $D$&$[D{}]^2$&$[D{}]\cdot [E]$\\
\hline
$1$&$-2K-[E]$&$\B^2$&$2$&$2$\\
\hline
\end{tabular}}
\end{lemma}

\begin{proof} It follows from $[D{}]^2=(-2K-r[E])^2=4-2r^2\ge-1$.
\end{proof}

\subsection{Proof of Proposition \ref{merging}}
Note first that our assumption $\chi(X^+_\R(t))<\chi( X^+_\R(-t))$
for $t>0$ implies that
the both rulings of $\til X^{0}(0)$ are real, and thus, the real structure descends from the family $\til X(\tau)$
to the family obtained by contraction of any of the two rulings.
When passing from $X^+(t)$  to the Bertini dual family $X^-(t)$, we do the same, only the direction is changing:
$\chi(X^-_\R(-t))<\chi( X^-_\R(t))$ if $t>0$.

To treat the semi-branches issued from a vertex we
apply the gluing procedure underlying Abramovich-Bertram-Vakil formula (see \cite[Proposition 4.1]{IKS} and \cite[Theorem 2.5]{Br-P}).
Enumeration of the semi-branches issued from a vertex represented by a curve $D$ is equivalent, by definition,
to enumerate curve families $A(t)$ having in the limit $A(0)=D{}+rE$, $r\ge 0$. To do this we use Lemmas \ref{e-walls}, \ref{v-walls}, and \ref{0-walls}.
Due to Lemmas \ref{e-walls}, \ref{v-walls},
and \ref{0-walls}, $r\le 2$ and for each $r\le 2$ there are 2 kind of cases to be considered separately: $D{}\cdot E=1$ and $D{}\cdot E=2$.

\subsubsection{\bf Case $\mathbf{r\le 1}$ and $\mathbf{D{}\cdot E=1}$}
At each vertex of such kind
we have 2 positive semi-branches. Namely, as it follows from Lemmas \ref{e-walls},
\ref{v-walls}, and \ref{0-walls},
a pair of real curves $A'(t,) A''(t)\in\CC_\R^{n,1}(X(t),x(t))$, $t>0$,
with $n=2$ or $4$ and $[A''_\R(t)]=[A'_\R(t)]+ [E_\R]$,
merge together in the limit $t\to 0^+$.
%
%
The
curves $A'(t)$, $A''(t)$
have the same number of cross-point nodes,
while
\begin{equation*}
q([A''_\R(t)])=q([A'_\R(t)]+ [E])=q([A'_\R(t)])+ q([E])+2= q([A'_\R(t)])+2.
\end{equation*}
This implies $\css(A'(t))+\css(A''(t)=0$.

There are no negative semi-branches at such a vertex,
since neither $[A'(0)]$
nor
$[A''(0)]$ is orthogonal to $[E]$, while the homology action of $\conj$ after wall-crossing is changed by reflection in $[E]$.

\subsubsection{\bf Case $\mathbf{r\le 2}$ and $\mathbf{D{}\cdot E=2, D{}^2=0}$}\label{conic-bundle-case}
In this case the curve $D$ generates a real rational ruling of $\til X^1(0)$. Thus, due to our choice as the base-point $x(0)\in X_\R(0)$ a point which is close to a generic point of $E_\R$ (see
assumption (4) in \ref{families-setting}), the curve $D$ intersects $E$ in 2 real points.
Therefore,  Lemmas \ref{e-walls}(2) and \ref{v-walls}(2)
provide 3 groups of real families of curves sharing such a
$D{}$ in the limit $t\to0^+$. They correspond to divisor classes of type $-2K-e-[E], -2K-e+[E]$ and $-2K -e$, $e\in R^2(X(t))$.
The first two classes provide one merging pair of curves $A'(t),A''(t)\in {\CC_\R^{4,2}(X(t),x)}$, $t\ge0$,
and the last class
yields another  merging pair $\til A'(t),\til A''(t)\in {\CC_\R^{2,0}(X(t),x)}$.

Equality of modulo 2 homology classes $[A'(t)]=[A''(t)]$, $[\til A'(t)]=[\til A''(t)]$,
implies
$\q([A'(t)])=\q([A''(t)])$ and $\q([\til A'(t)])=\q([\til A''(t)])$. On the other hand, $A'(t)$ and $A''(t)$ has no singular points at all, while the number of cross-points in $\til A'_\R(t),\til A''_\R(t)$ is by 1 greater than
the number of
cross-points in $D_\R$, that is in $A'_\R(0)$. Therefore,
$$\css(A'(t))=\css( A''(t))=-\css(\til A'(t))=-\css(\til A''(t))$$
 as it follows from the definition of weights
(\ref{s-weight-definition}).
This implies that the sum $\css(A'(t))+\css( A''(t))+\css(\til A'(t))+ \css(\til A''(t))$ that corresponds to positive semi-branches issued from $D$ is zero.

The situation with counting by means of the weights (\ref{ss-weight-definition}) is different.
Here we need to treat the branches of type $(4,2)$ and type $(2,0)$ separately. For that, we use a bijection between the set of vertices $D$ under consideration and the set of real rulings with $D\cdot E=2$, and, in its turn,
with the real roots $e$ orthogonal to $[E]$.
Thus, taking the sum over this set of vertices
we get
$$\sum_{A\in\CC_\R^{4,2}(X(t),x(t))}\css(A)=
\sum_{e\in R^2(t),\, e\cdot E=0}\q(e)=2(r^+-1),$$
and
$$\sum_{A\in\CC_\R^{2,0}(X(t),x(t))}\css(A)=
-\sum_{e\in R^2(t), e\cdot E=0}\q(e)=-2(r^+-1).$$
There are no negative semi-branches at this kind of vertices. Indeed,
there are no real curves in divisor classes $-2K-e-[E], -2K-e+[E]$ on $X(-t)$ since these classes
are not orthogonal to $[E]$, and there are no negative semi-branches for $-2K -e$, since,
as it was already pointed out in the beginning, all the counted rational curves of class $-2K-e$ share in the limit the curve $D\subset \til X^1_\R(0)$
that intersects $E_\R\subset \til X^1_\R(0)$ at 2 real points, which obstructs gluing $D$ with 2 generators of $\til X^0(0)$ to have a real perturbation in $X(-t)$
({\it cf.} \cite[Theorem 2.5]{Br-P}).

\subsubsection{\bf Case $\mathbf{r\le 2, D{}^2=0}$ and $\mathbf{D{}\cdot E=2}$}\label{real-pair}
Here, we distinguish two sub-cases: the 2 points of $D\cap E$ are real or imaginary complex conjugate.

First, assume that the both points where $D{}$ meets $E$ are real. Then
 Lemmas \ref{e-walls}(3) and \ref{0-walls}
provide four families of curves sharing the same divisor $D{}$ in the limit $t\to0^+$.
Two of them represent
divisor classes $-2K-[E]$, $-2K+[E]$ and the other two belong to the same class $-2K$.
The first two classes provide one merging pair of curves $A'(t),A''(t)\in {\CC_\R^{2,2}(X(t),x(t))}$, $t\ge0$,
and the last class contains another  merging pair $\til A'(t),\til A''(t)\in {\CC_\R^{0}(X(t),x(t))}$.

Note that $\q$ vanishes for all these curves, since
$$
\q([-2K])=0, \, q([-2K\pm E])=\q([E])=0 $$
and, thus, $\css$ for the curves involved
coincides with the weights
$$w(A'(t))=w(A''(t))=-w(\til A'(t))=-w(\til A''(t)),$$
where alternation is explained by the same reason as in the previous case. This gives $\css(A'(t))+\css(A''(t))+\css(\til A'(t))+\css(\til A''(t))=0$ for $t>0$. Since, also as above, there are no negative semi-branches issued from
$D$, we are done with part (A) of Proposition in the case of real $D\cap E$.

If the 2 points of $D\cap E$ are imaginary complex conjugate, then we have two positive semi-branches $A'(t),A''(t)$ of divisor classes $-2K-[E]$, $-2K+[E]$ and two negative semi-branches $\til B'(-t),\til B''(-t)$ of divisor class $-2K$.
The curves $\til B'(-t),\til B''(-t)$ acquire an additional,  as compared with $A'(t),A''(t)$, solitary point but have the same number of cross-points.
Therefore, we have still a balance between positive and negative semi-branches: $\css(A'(t))+\css(A''(t))=\css(\til B'(t))+\css(\til B''(t)).$

For completing the proof of part (B) of Proposition, it remains
to evaluate the sum of $\css$-weights over all positive semi-branches in a fixed class $-2K\pm[E]$.
Such a sum can be counted like in (\ref{euler-integration}), in the proof of Proposition \ref{First-summand}.
Each class yields a sum equal to $\chi(X_\R(t))-1=r^+-r^-$, and both classes together give $2(r^+-r^-)$. \qed

\section{Concluding remarks}\label{Concluding}

\subsection{Count of quartics 6-tangent to a sextic on a quadratic cone $Q$}\label{6-tangents}
For a given sextic $C\subset Q$,
a {\it 6-tangent quartic} is a transversal intersection $A=Z\cap Q $ of $Q$ with a quadric $Z$ such that
the intersection divisor
$Z \circ C=A\circ C$
contains each point with even multiplicity. 
Let us denote 
by $T(Q,C,z)$ the set of rational irreducible reduced
6-tangent quartics
that  pass  through a fixed point $z\in Q\sm C$. Consider, as usual, the double covering $\pi : X\to Q$ branched along $C$.
 For a generic $z\in Q\sm C$ and $\pi^{-1}(z)=\{x_1,x_2\}$, each of the induced projection
$\pi_k: \CC^2(X,x_k)\cup\CC^4(X,x_k)\to T(Q,C,z), k\in\{1, 2\},$ is a bijection.

Over $\R$, we pick a pair of real points $z^\pm\in Q^\pm_\R=\pi(X^\pm_\R)$ and 
consider two
sets of real 6-tangent quartics, $T_\R(Q_\R^\pm,C, z^\pm)=\{A\in T(Q,C,z^\pm)\,|\, A \text{ is real}\} $.
Let $\pi^{-1}(z^\pm)=\{x^\pm_1,x^\pm_2\}$.
Note that the real locus $A_\R$ of each
 $A\in T_\R(Q_\R^\pm,C, z^\pm)$ lies entirely inside region $Q^\pm_\R$,
and, thus, for each choice of $k=1,2$, the curve $A$
lifts to a unique real rational curve $A_k=\pi_k^{-1}\in  \CC^2_\R(X^\pm, x^\pm_k)\cup \CC^4_\R( X^\pm, x^\pm_k)$, 
and we obtain bijections
$$\CC^2_\R( X^\pm, x^\pm_k)\cup \CC^4_\R(X^\pm, x^\pm_k) \to T_\R(Q^\pm,C,z^\pm).$$

By Theorem \ref{main-Pin}(1),
$\css(A_1)=\css(A_2)$ for each $A\in T_\R(Q_\R^\pm,C, z^\pm)$.
This allows to distinguish
two species of real 6-tangent quartics:
{\it hyperbolic} if $\css(A_k)>0$, and {\it elliptic} if $\css(A_k)<0$.
Recall, that by definition $\css(A_k)$ takes the values $\pm 1$.

We also consider a larger set
$\til T_\R(Q^\pm,C,z^\pm)\supset T_\R(Q^\pm,C,z^\pm)$
that
in addition to 6-tangent quartics contains those real conics in $Q_\R^\pm$ that are 2-tangent to $C$
and pass through $z^\pm$. For every $k=1,2$, each of these conics $A$ lifts to a
unique curve
$\til A\in \CC^0_\R( X^\pm, x_k^\pm)$
which provides us with a bijection
$$\CC^0_\R( X^\pm, x^\pm_k)\cup \CC^2_\R( X^\pm, x^\pm_k)\cup \CC^4_\R( X^\pm, x^\pm_k)\to \til T_\R(Q^\pm,C,z^\pm).$$
Similar to above,
we call a real conic $A\in \til T_\R(Q^\pm,C,z^\pm)$ hyperbolic if $\css(\til A)>0$ and elliptic if $\css(\til A)<0$.
We let also $\cs(A)=\cs(A_1)= \cs(A_2)$ for every $A\in T_\R(Q^\pm,C,z^\pm)$.

Theorems \ref{main-30} and \ref{main} imply then the following result.

\begin{theorem} Assume that a real sextic $C\subset Q$ is a transversal intersection of a
real quadratic cone $Q\subset \P^3$ (whose base is non-singular and
has non-empty real locus)
with a real cubic surface. Then, the following holds for any generic pair of points $z^\pm\in Q^\pm_\R$:
\begin{enumerate}
\item
The number of hyperbolic minus the number of elliptic elements in
each of the set $\tilde T_\R(Q^\pm,C, z^\pm)$ is 30.
\item
The number of real quartics $A$ that are 6-tangent to $C$ {\rm (}{\it i.e.} of elements in $T_\R(Q^+,C,z^+)\cup T_\R(Q^-,C,z^-)${\rm )} counted with weight
$\cs(A)$ is 96.
\end{enumerate}
\end{theorem}

\subsection{A few speculations on extension to symplectic setting}
A natural question is how
to extend the
results obtained in \cite{TwoKinds} and in this paper
to
{\it symplectic del Pezzo surfaces of degree 1}, that is to
rational symplectic 4-manifolds $(X,\omega_X)$ with $K^2_X=1$.
Alas, even for the existence of an appropriate $Pin^-$-structure, additional conditions will have to be imposed
(see Proposition \ref{main-symplectic})
Namely, we restrict our consideration to those $(X,\omega_X)$ that are
equipped with an analog of {\it Bertini involution},  that is a symplectic involution $\tau_X: X\to X$
which acts in $H^2(X)$ as a reflection in $K_X$,
and for which
there exist $\omega_X$-tamed $\tau_X$-invariant generic $J$-holomorphic structures $J_X$
such that the following {\it Bertini condition} is satisfied:
all $J$-holomorphic curves Poincar\'e dual to $K_X$ have a common point.

Due to these assumptions, the above curves form a symplectic Lefschetz pencil
composed of curves of arithmetic genus 1, and
then $\tau_X$ can be seen as an extension by continuity of the involution that acts
on non-singular elements of the pencil
as multiplication by $(-1)$ with respect to the group structure with $0$ in the fixed point of the pencil.
Taking the quotient of $X$ by
$\tau_X$, we obtain a symplectic quadratic cone
(defined uniquely up to symplectic deformation): an orbifold symplectic manifold $Q$ with one node $q$ (the image of the above fixed piont $p$) and an orbifold Lefschetz pencil formed by curves of genus 0
passing through the node.

All this
can be extended to the real setting, that is to symplectic manifolds equipped with an anti-symplectic involution. So, starting from a real rational symplectic 4-manifold $(X,\omega_X, J_X, \tau_X, \conj_X)$ with $K_X^2=1$, $(X, \omega_X, J_X, \tau_X)$ satisfying the above Bertini condition, and an anti-symplectic involution $\conj_X$ commuting with $\tau_X$, we get
a real symplectic quadratic cone $(Q,\omega_Q,\conj_Q)$ with a real symplectic quotient map $\pi: X\to Q$. The real part of $Q$ is then diffeomorphic to the real part of a real algebraic quadratic cone over a real non-empty nonsingular conic,
and $\pi_\R: X_\R\sm p \to Q_\R \sm q$ naturally turns into a composition of an embedding of $X_\R$ into a real line bundle $\xi$ over $Q_\R\sm q$ with the projection from the total space of the bundle to its base.
The bundle $\xi$ is nontrivial, its Stiefel-Whitney class $w_1(\xi)$ is Poincar\'e dual to a line-generator of $Q_\R$. In particular, this gives a natural diffeomorphism between the total space of $\xi$ and $(\Rp2\sm point) \times \R$.

Thus, applying this construction we get an embedding $X_\R\sm p \to \Rp2\times \R$ defined uniquely up to isotopy and reversing of $\R$-direction in $\Rp2\times \R$. After that, we proceed as in \cite{TwoKinds}.
The 3-manifold $\Rp2\times \R$ admits a unique (up to reversing)
$\Pin^-$-structure, the latter induces a unique (up to reversing)
$\Pin^-$-structure on $X_\R\sm p$, which in its turn extends to $X_\R$.
Finally,  between the two opposite
structures we choose that one (which we called {\it monic}) whose quadratic function $q_X$ takes value $1$ on $w_1(X_\R)$.

As a conclusion, we obtain, in particular, the following statement.

\begin{proposition}\label{main-symplectic}
The above construction supplies each real symplectic del Pezzo surface $(X, \omega_X, J_X,\tau_X, \conj_X)$  of degree 1
satisfying Bertini condition
with a $\Pin^-$-structure $\theta_{X}$ on $X_\R$, so that the following properties hold:
\begin{enumerate}
\item\label{invar}
$\theta_X$
is invariant under real automorphisms and real deformations of $(X,\omega_X, \\ J_X, \tau_X, \conj_X)$ that respect Bertini condition. In particular, the quadratic function
$q_{X} : H_1(X_\R;\Z/2)\to\Z/4$ associated with $\theta_X$ is preserved by the Bertini involution.
\item\label{vanish}
$q_{X}$
takes value $0$ on each cycle in $H_1(X_\R;\Z/2)$ which is vanishing under deformations respecting Bertini condition and
takes value $1$ on the class dual to $w_1(X_\R)$.
\item\label{sym}
If $X^\pm$ is
a Bertini pair of real symplectic del Pezzo surfaces of degree 1, then
the corresponding quadratic functions $q_{X^\pm}$
take equal values on the elements represented
in $H_1(X^\pm_\R;\Z/2)$  by the connected components of $C_\R$.
\end{enumerate}
\end{proposition}

It looks to us plausible that the following conjecture should be true, and that its proof can be achieved by the methods borrowed from \cite{Seva}, \cite{withSeva}.

\newtheorem{conject}[theorem]{Conjecture}
\begin{conject} Each deformation class of real symplectic del Pezzo surfaces  of degree 1 equipped with $J$-holomorphic structure
satisfying Bertini condition
contains a real algebraic del Pezzo surface of degree 1.
\end{conject}

If this conjecture will turn out to be true, it will imply that the enumerative results obtained in this paper extend to the symplectic setting described above.

\end{document}